\documentclass[11pt,oneside,a4paper]{article}
\usepackage{lmodern}
\usepackage{babel}
\usepackage[T1]{fontenc} 
\usepackage[utf8]{inputenc} 
\usepackage{verbatim}  
\usepackage[a4paper]{geometry} 
\usepackage{amsmath} 
\usepackage{graphicx} 
\graphicspath{{images/}} 
\usepackage{systeme}
\usepackage{amssymb}
\usepackage{amsfonts}
\usepackage{csquotes}
\usepackage{hyperref}
\usepackage{algorithmic}
\usepackage{algorithm}

\usepackage{amsmath,amssymb}
\usepackage{amsthm}

\newtheorem{theorem}{Theorem}[section]
\newtheorem{conjecture}{Conjecture}[section]
\newtheorem{definition}{Definition}[section]
\newtheorem{lemma}{Lemma}[section]

\newtheorem{remark}{Remark}[section]
\newtheorem{remarks}{Remark}[section]
\newtheorem{examples}{Examples}[section]
\newtheorem{example}{Example}[section]

\usepackage{amssymb}
\usepackage{amsmath}
\usepackage{amstext}

\usepackage{comment}

\newcommand*\modd[3]{#1\equiv#2\;( {\rm mod}\;  #3)}
\newcommand*\notmodd[3]{#1\not\equiv#2\;( {\rm mod}\;  #3)}

\title{Weakly and Strongly Admissible Triplets for a Collatz-Type Map}

\author{Abderrahman  Bouhamidi\footnote{L.M.P.A, Universit\'e du
		Littoral, 50 rue F. Buisson BP699, F-62228 Calais-Cedex, France.  \href{mailto:abderrahman.bouhamidi@univ-littoral.fr}{abderrahman.bouhamidi@univ-littoral.fr}}}

\date{November 11,  2025}

\begin{document}
\maketitle

\begin{abstract}
In this paper, we investigate a class of Collatz-like problems associated with weakly and strongly admissible triplets of integers. This framework extends the classical Collatz mapping, providing a systematic method for generating triplets with convergence to cycles, thereby bypassing the difficulties inherent in solving Diophantine equations.
We introduce several special families of admissible triplets and establish general structural properties.
In addition, we propose conjectures that generalize the classical Collatz conjecture. Bounds on the lengths of potential non-trivial cycles are derived and analyzed, and two algorithms are presented for computing lower bounds on cycle lengths. Finally,  experimental  tests are given to illustrate our approach.
\end{abstract}

\section{Introduction}
The classical Collatz problem  also known as   Kakutani, Syracuse or  $3n+1$ problem, concerns the sequence of positive integers generated by the  iterations of the  Collatz mapping  $C:\mathbb{N}\longrightarrow \mathbb{N}$ defined as:
\begin{equation}\label{CollatzMap}
	C(n)=\left\{\begin{array}{lcc}
		n/2&\mathrm{if}& n\;\mathrm{is\; even}\\ 
		(3n+1)/2&\mathrm{if}& n\;\mathrm{is\; odd}, \\
	\end{array}\right.
\end{equation}
where $\mathbb{N}=\{1,2,3,\ldots\}$  is the set  of integers $\geq 1$. For every $k\in\mathbb{N}_0:=\mathbb{N}\cup\{0\}$, the notation $C^{(k)}$ stands for the $k$-th iterate of the mapping $C$.  The well known conjecture of the $3n+1$ problem  asserts that: Starting with  any initial integer $n\in\mathbb{N}$, there is a positive integer $k\in\mathbb{N}_0$ such that $C^{(k)}(n)=~1$. 
The simplicity of its  formulation makes this conjecture very popular. Although this problem is very simple to formulate, it is well known that finding a rigorous and   a complete proof of it  remains an open problem challenge  in mathematics that require solid arguments \cite{Lagarias1985}.  Despite this difficulty, many interesting results are established in the literature and several  generalizations of the  Collatz mappings have arisen since many  years \cite{Allouche1978,MatthewsWatts1984,MatthewsWatts1985,Lagarias1990,Eliahou1993}.  Relevant literature and  extensive  surveys as well as historical  discussion  on this subject may be found in  \cite{Lagarias1985,Lagarias2011}.     In an attempt to understand the outlines of a possible proof of such a conjecture, many authors have studied integer decomposition techniques involving the numbers 2 and 3. Others have instead tried to give generalizations of the conjecture that will perhaps allow general properties to be derived. In this article, we have chosen to follow the second direction. So, let us start our paper by given a brief  and non exhaustive overview on some generalizations of the Collatz problem previously studied   in the literature.  For more details, we refer the reader to \cite{Lagarias1985,Lagarias2011, Wirsching1998}.  Before that,  let us introduce the following notation, for a positive integer   $d\geq 2$, the notation $[n]_d$  denotes the remainder of the  integer $n$ in the Euclidean division of $n$ by $d$ with the standard   condition   $0\leq [n]_d<d$.  We recall that  for all positive integer $n$ we have  $[-n]_d=d-[n]_d$.  Now, let us recall the following knowen  generalizations of the Collatz map:

\begin{itemize}
	\item In    1979, Allouche \cite{Allouche1978} proposed  the following generalization, see  the resume in \cite{Lagarias2011},
	$$
	A_{d,m}(n)=\left\{\begin{array}{ccc}
		n/d&\mathrm{if}& \modd{n}{0}{d}\\
		(m n-m [n]_d)/d&\mathrm{if}& \notmodd{n}{0}{d},\\
	\end{array}\right.
	$$
	in which the parameters $d$ and $m$ satisfy $d\geq 2$, $gcd(m,d)=1$.  The generalization of Allouche is an extension to  the work introduced previously by Hasse.   The author noted that it is easy to show that any
	mapping in Hasse’s class with $1 \leq  m < d$ has a finite number of cycles, and for this mapping
	all orbits eventually enter one of these cycles.    
	\item  In 1984,  Matthews and Watts \cite{MatthewsWatts1984,MatthewsWatts1985}, see also
	\cite{Matthews2010},  proposed a complete generalization  to the classical Collatz  problem which  may be seen as an extension of the mapping previously proposed by Hasse.   The generalization  of Matthews and Watts, may be  summarized briefly  as following: Let $d\geq 2$ be a positive integer and $m_0,\ldots,m_{d-1}$ be non-negative integers and let 
	$r_0,\ldots,r_{d-1}$ be  relative integers satisfying the conditions $\modd{r_i}{im_i}{d}$, for $i=0,\ldots,d-1$.  Thus, the general mapping denoted here by $M:\mathbb{Z}\longrightarrow \mathbb{Z}$ and 
	defined in \cite{MatthewsWatts1984,MatthewsWatts1985} is given by the  formula
	\begin{equation}\label{MatthewsMap}
		M(n)=\dfrac{m_in-r_i}{d}\;\;\textrm{if}\; \modd{n}{i}{d}.
	\end{equation}
	\item In 1990,  Lagarias \cite{Lagarias1990},  proposed the following map
$L_\beta(n)=\left\{\begin{array}{ccc}
	n/2&\mathrm{if}& n\;\mathrm{is\; even}\\
	(3 n+\beta)/2&\mathrm{if}& n\;\mathrm{is\; odd},\\
\end{array}\right.$
with  $\beta\geq 1$ is an integer  such that $gcd(\beta,6)=1$.   
\end{itemize}
In this  paper, we will restrict our attention  only to the mappings acting  from $\mathbb{N}$ into $\mathbb{N}$.   The first conjecture we introduce  here is stated in simple manner as following:
\begin{conjecture}\label{Myconjecture}
	Starting with a positive integer  $n\geq 1$:
	\begin{itemize}
		\item[$\bullet$]  If $n$  is a multiple of $10$  then remove  the zero on the right.  Otherwise, 
		multiply it by $6$, add  $4$ times  its  last digit and divide the result  by $5$.
		\item Repeat the process infinitely.
	\end{itemize}
	Then,  regardless the starting number, the process eventually reaches $4$  after a finite number of iterations.
\end{conjecture}
More details on this conjecture are summarized  in \cite{Bouhamidi_conj10.6.4}.  The previous conjecture is in fact associated to the operator
$T:~\mathbb{N}~\longrightarrow~\mathbb{N}$ given by
\begin{equation}\label{mapT0}
	T(n)=\left\{\begin{array}{ccc}
		n/10&\mathrm{if}& \modd{n}{0}{10}\\
		(6 n+4 [ n]_{10})/5&\mathrm{if}& \notmodd{n}{0}{10},\\
	\end{array}\right.
	\; \mbox{or}\;
	T(n)=\left\{\begin{array}{ccc}
		n/10&\mathrm{if}& \modd{n}{0}{10}\\
		(12 n+8 [ n]_{10})/10
		&\mathrm{if}& \notmodd{n}{0}{10}.\\
	\end{array}\right.
\end{equation}

In this paper, we present a natural extension of the classical Collatz conjecture, as well as our proposed Conjecture \ref{Myconjecture}, formulated through the  concept of weakly and strongly admissible triplets of integers. By focusing on admissible triplets rather than individual integers, we are able to uncover new structural properties and families of sequences exhibiting predictable behavior.  This study not only offers new insights into Collatz-type dynamics but also provides tools that may facilitate further exploration of integer sequences with similar properties, opening avenues for both theoretical and computational investigations. The outline of this paper is as follows.
In Section 2, we introduce the concept of admissible triplets and present their main definitions and properties. Section 3 focuses on special classes of admissible triplets that yield trivial cycles, and several general conjectures are formulated. In Section 4, we study some properties concerning lower bounds for possible  non-trivial cycle lengths. Two algorithms will be derived from specific  theoretical results. In the last section, we will present  few experimental examples  illustrating the announced  conjectures and the theoretical study.

\section{Weakly  and strongly admissible  triplets of integers -  Definitions}\label{SectionGeneralization}
Let $d \geq2$ be a positive integer number and consider a pair $(\alpha,\beta)$ of two other integers with $\alpha>d$  and  $\beta$ may be negative. The integers  $d$, $\alpha$ and $\beta$  are  such that:   $\notmodd{\alpha}{0}{d}$ and  $\notmodd{\beta}{0}{d}$.  In the following definition, we allow $gcd(\alpha,d)$ or $gcd(\beta,d)$ to be great then $1$. A natural and obvious  extension to the classical Collatz mapping  is  the mappings $T:\mathbb{N}\longrightarrow \mathbb{N} $ given  by
\begin{equation}\label{mapT}
	T (n)=\left\{\begin{array}{ccc}
		n/d&\mathrm{if}& \modd{n}{0}{d}\\
		(\alpha n+\beta [\kappa_0 n]_d)/d&\mathrm{if}& \notmodd{n}{0}{d},\\
	\end{array}\right.
\end{equation}
where the parameter $\kappa_0=\pm1$. The triplet  $(d,\alpha,\beta)$   will be  denoted by $(d,\alpha,\beta)_{\pmb{+}}$ and 
$(d,\alpha,\beta)_{\pmb{-}}$ for the case $\kappa_0=+1$ and for the case  $\kappa_0=-1$, respectively.   The mapping  \eqref{mapT} encompasses obviously the classical Collatz mapping  \eqref{CollatzMap} together with the introduced mapping in \eqref{mapT0}.  We have the following obvious result.

\begin{theorem}
	A necessary and sufficient condition that the 
	mapping
	$T:\mathbb{N}\longrightarrow \mathbb{N} $  given by \eqref{mapT} 
	is well defined from  $\mathbb{N}$ into $\mathbb{N}$ is that the triplet 
	$(d,\alpha,\beta)_{\pmb{\pm}}$ associated to $T$    satisfy the following condition
	\begin{equation}\label{CdtCollatztriplet}
		\alpha+\kappa_0\beta>\dfrac{(\kappa_0-1)}{2}\beta\,d \quad\mathrm{and}\quad \modd{\alpha+\kappa_0\beta}{0}{d}.
	\end{equation}
\end{theorem}
\begin{proof}
	Let   $n\in\mathbb{N}$ and  consider its Euclidean division by $d$ of the form $n=q_nd+[n]_d$.  If  the remainder  $[n]_d=0$, then $T(n)= q_n\in \mathbb{N}$.  Else,  $[n]_d\not=0$ and 
	$$T(n)=(\alpha n+\beta [ \kappa_0n]_d)/d=
	\left\{\begin{array}{ccc}
		\alpha q_n+\dfrac{\alpha+\beta}{d}[n]_d&\mathrm{if}& \kappa_0= +1\\
		\alpha q_n+\beta+\dfrac{\alpha-\beta}{d}[n]_d&\mathrm{if}&  \kappa_0= -1.\\
	\end{array}\right.$$
	For $\kappa_0=+1$,  if the condition
	\eqref{CdtCollatztriplet}  is satisfied then 
	$\alpha+\beta>0$ and  $\modd{\alpha+\beta}{0}{d}$, it follows that
	for all $n\in\mathbb{N}$, we have $T(n)\in\mathbb{N}$.\\ Now,  for  $\kappa_0=-1$,  if the condition
	\eqref{CdtCollatztriplet}  is satisfied then 
	$(\alpha-\beta)+\beta d>0$ and  $\modd{\alpha-\beta}{0}{d}$. Then  for all $n\in\mathbb{N}$, we have $T(n)\in\mathbb{Z}$.
	If $\alpha-\beta>0$, as  $[n]_d\geq 1$ then $T(n)\geq \beta+\dfrac{\alpha-\beta}{d}[n]_d\geq \beta+\dfrac{(\alpha-\beta)}{d}=\dfrac{\beta d+(\alpha-\beta)}{d}> 0$. If $\alpha-\beta<0$, as  $[n]_d\leq (d-1)$ then $T(n)\geq \beta+\dfrac{\alpha-\beta}{d}[n]_d\geq \beta+\dfrac{(\alpha-\beta)}{d}(d-1)=\dfrac{\alpha(d-1)+ \beta}{d}$. In this case, we have $\beta >\alpha\geq d\geq 2$. This shows  that also for $\kappa_0=-1$, we have
	for all $n\in\mathbb{N}$, $T(n)\in\mathbb{N}$.
 Conversely, if $T:\mathbb{N}\longrightarrow \mathbb{N} $ is well defined, then  $T(1)\in \mathbb{N}$ and 
 we have
	$$T(1)=
	\left\{\begin{array}{ccc}
		\dfrac{\alpha+\beta}{d}&\mathrm{if}& \kappa_0= +1\\
		\dfrac{\alpha +\beta(d-1)}{d}&\mathrm{if}&  \kappa_0= -1,\\
	\end{array}\right.$$ 
	it follows 
	obviously that $\modd{\alpha+\kappa_0\beta}{0}{d}$ for both cases and we have $\alpha+\beta>0$
	for
	$\kappa_0=1$ and 
	$\alpha-\beta>-\beta d$ for $\kappa_0=-1$. Thus, the condition \eqref{CdtCollatztriplet} holds.
	\hfill\end{proof}

	As an example, the triplet   $(2,3,1)_{\pmb{+}}$  corresponding to the  classical   Collatz problem satisfies the condition \eqref{CdtCollatztriplet}. 
	The triplet $(3, 5,2)_{\pmb{-}}$  satifies \eqref{CdtCollatztriplet} but $(3,5,2)_{\pmb{+}}$ does  not. The associated  mapping  $T:\mathbb{N}\longrightarrow \mathbb{N} $  to $(3,5,2)_{\pmb{+}}$ given  by
	$$
	T(n)=\left\{\begin{array}{ccc}
		n/3&\mathrm{if}& \modd{n}{0}{3}\\
		(5 n+2 [ n]_3)/3&\mathrm{if}& \notmodd{n}{0}{3},\\
	\end{array}\right.
	$$
	is not well defined on $\mathbb{N}$. Indeed,  for instance $T(1)=\dfrac{7}{3}\not\in\mathbb{N}$.  The triplet $(5,7,-3)_{\pmb{-}}$ does not also satisfy the condition \eqref{CdtCollatztriplet} since  the condition  $\alpha+\kappa_0\beta>\dfrac{(\kappa_0-1)}{2}\beta\,d$  does not hold and for the associated $T$, we have for instance $T(1)=-1 \not\in\mathbb{N}$.

\begin{remark}
	In fact,  if   $(d,\alpha,\beta)_{\pmb{\pm}}$ satisfies the condition \eqref{CdtCollatztriplet}, then   the mapping  given in \eqref{mapT} is a special case of the general case given in \eqref{MatthewsMap}. Indeed, 
	\begin{itemize}
		\item for $\kappa_0=1$ and from  \eqref{CdtCollatztriplet}, we have  $\modd{\alpha+\beta}{0}{d}$, then 	if we choose in \eqref{MatthewsMap}, $ m_0= 1$, $r_0=0$ and  for $i=1,\ldots,d-1$, we set
		$m_i=\alpha$ and  $r_i=-i\beta$,  it follows that $\modd{r_i}{i\alpha}{d}$ which gives the conditions $\modd{r_i}{im_i}{d}$, for $i=1,\ldots,d-1$.
		\item for $\kappa_0=-1$ and from  \eqref{CdtCollatztriplet}, we have  $\modd{\alpha-\beta}{0}{d}$, then 	if we choose in \eqref{MatthewsMap}, $ m_0= 1$, $r_0=0$ and  for $i=1,\ldots,d-1$, we set
		$m_i=\alpha$ and  $r_i=-\beta(d-i)$,  it follows that $\modd{r_i}{-(d-i)m_i}{d}$ which gives the conditions $\modd{r_i}{im_i}{d}$, for $i=1,\ldots,d-1$.
	\end{itemize}
	Although the mapping defined by \eqref{mapT} fits within the general framework of Matthews and Watts \cite{MatthewsWatts1984,MatthewsWatts1985}, our formulation offers a simpler and more explicit representation that supports concrete computations and a systematic study of Collatz-type dynamics. It enables the direct construction of trivial cycles, the derivation of explicit iteration formulas, and the efficient computation of total stopping times. This perspective also facilitates the identification of weakly and strongly admissible triplets, leading to conjectures extending the classical Collatz problem—Contributions that are not directly addressed in the original Matthews–Watts framework. The notions of weakly ans strongly admissible triplet will be introduced in the next definition.
\end{remark}

Let us now recall some classical definitions. Given $n\in \mathbb{N}$, the  trajectory   of $n$ is the set $\Gamma(n)$ of the successive iterates starting from $n$: $\Gamma(n)=\bigl
\{n,T(n),T^{(2)}(n),\ldots\bigr\}$.  A cycle (if it exists),  having $k$ elements and  associated to the mapping  $T$ (or to the triplet $(d,\alpha,\beta)_{\pmb{\pm}}$)  is a finite set  $\Omega $  for which $T^{(k)}(x)=x$ for all $x\in \Omega$.  The cardinality $\#\Omega=Card(\Omega)=k$, also called the length of  the cycle $\Omega$,  is  the smallest integer  $k$ such that $T^{(k)}(x)=x$ for all $x\in \Omega$.  We will denote by  
$\Omega(\omega)=\{\omega,T(\omega),T^{(2)}(\omega),\ldots, T^{(k-1)}(\omega)\}$
the  cycle (if it exists) associated to $T$  of length  $k$
where $\omega$ is the smallest element in  the cycle. 
We will also use this abusive following  notation 
$$(\omega \rightarrow T(\omega)\rightarrow T^{(2)}(\omega)\rightarrow \ldots \rightarrow T^{(k-1)}(\omega)\rightarrow \omega),$$
to denote  the  cycle $\Omega(\omega)$. 
For each choice of the  triplet   $(d,\alpha,\beta)_{\pmb{\pm}}$,   the notation $\mathrm{Cycl}(d,\alpha,\beta)_{\pmb{\pm}}$ stands for the set of all possible cycles associated to $T$, it may be empty, finite or infinite set, then   the cardinality
$\#\mathrm{Cycl} (d,\alpha,\beta)_{\pmb{\pm}}\in
\mathbb{N}_0 \cup\{\infty\}$, 
will be called the order of the triplet  $(d,\alpha,\beta)_{\pmb{\pm}}$ and will be denoted by
$\# \mathrm{Cycl}(d,\alpha,\beta)_{\pmb{\pm}}:=\#\mathrm{Cycl} (d,\alpha,\beta)_{\pmb{\pm}}$.
For $n\in\mathbb{N}$,  there are two possible situations for a trajectory $\Gamma(n)=\{T^{(k)}(n)\}_{k\in\mathbb{N}_0}$ starting with $n$:
\begin{enumerate}
	\item[(i)]  Convergent trajectory to  either  cycle, assuming that   such  cycles exist.
	\item[(ii)] Divergent trajectory: $\displaystyle \lim_{k\rightarrow \infty}T^{(k)}(n)=\infty$.
\end{enumerate}
Let a cycle $\Omega(\omega)$ exist, we denote by $G_\omega(d,\alpha,\beta)_{\pmb{\pm}}$ and $G_\infty(d,\alpha,\beta)_{\pmb{\pm}}$ the subsets of $\mathbb{N}$  given by
$$G_\omega(d,\alpha,\beta)_{\pmb{\pm}}=\{n\in\mathbb{N}:\, \exists k\in\mathbb{N},\,T^{(k)}(n) \in \Omega(\omega)\}$$
and  $$G_\infty(d,\alpha,\beta)_{\pmb{\pm}}=\{n\in\mathbb{N}:\, \lim_{k\rightarrow \infty} T^{(k)}(n)=\infty\}.$$
The set $G_\infty(d,\alpha,\beta)_{\pmb{\pm}}$ may be empty or not and the sets  $G_\omega(d,\alpha,\beta)_{\pmb{\pm}}$ are pairwise  disjoint: 
For all $\omega,\omega'\in\mathbb{N}$ such that the  cycles $\Omega(\omega)$ and  $\Omega(\omega')$ exist, if  $\omega\not= \omega'$, then  $ G_\omega(d,\alpha,\beta)_{\pmb{\pm}}\cap G_{\omega'}(d,\alpha,\beta)_{\pmb{\pm}} =\emptyset$ and 
$G_\infty(d,\alpha,\beta)_{\pmb{\pm}}
\cap  G_\omega(d,\alpha,\beta)_{\pmb{\pm}}=\emptyset$.  

The  condition  \eqref{CdtCollatztriplet}  does not ensure  the existence of at least   one cycle.  For instance   the  triplets	$(2,9,1)_{\pmb{+}}$, 
$(5,31,4)_{\pmb{+}}$ 
and $(15,35,10)_{\pmb{+}}$ 
although they satisfy the condition  \eqref{CdtCollatztriplet}, it appears that they do not have any cycle. 
We have   the following definition.

\begin{definition}
	Let $d \geq2$ be a positive integer and consider a pair $(\alpha,\beta)$ of two other integers with $\alpha>d$  and  $\beta$ may be negative. The integers  $d$, $\alpha$ and $\beta$  are  such that   $\notmodd{\alpha}{0}{d}$ and $\notmodd{\beta}{0}{d}$ and   its associated  mapping   $T$ given by  \eqref{mapT}.  We will say that:
	\begin{itemize}
		\item 
			The triplet $(d,\alpha,\beta)_{\pmb{\pm}}$   is an admissible triplet of integers 
			if and only if  it
			satisfies the condition  \eqref{CdtCollatztriplet} and  $(d,\alpha,\beta)_{\pmb{\pm}}$ has at least one cycle. Thus $1\leq \# \mathrm{Cycl}(d,\alpha,\beta)_{\pmb{\pm}}$.

		\item    The triplet $(d,\alpha,\beta)_{\pmb{\pm}}$  is a weakly  admissible   triplet  if and only if  $1\leq \# \mathrm{Cycl}(d,\alpha,\beta)_{\pmb{\pm}}<\infty$. Namely, $(d,\alpha,\beta)_{\pmb{\pm}}$  is admissible and
 possesses a finite  number of cycles. 
		\item  The triplet $(d,\alpha,\beta)_{\pmb{\pm}}$  is a  strongly  admissible triplet  if and only if $(d,\alpha,\beta)_{\pmb{\pm}}$  is a weakly  admissible   triplet and 
		 $G_\infty(d,\alpha,\beta)_{\pmb{\pm}}=\emptyset$. Namely, $(d,\alpha,\beta)_{\pmb{\pm}} $
		  possesses a finite, non-zero number of cycles without any divergent trajectory.
	\end{itemize} 
\end{definition}

According to the previous definition, it follows that:
\begin{itemize}
	\item If $(d,\alpha,\beta)_{\pmb{\pm}}$  is a
weakly admissible  triplet, then $(d,\alpha,\beta)_{\pmb{\pm}} $
	has  a   finite  number $q=\# \mathrm{Cycl}(d,\alpha,\beta)_{\pmb{\pm}} $ of cycles denoted as  $\Omega(\omega_1),\ldots, \Omega(\omega_{q})$ with  $1\leq q<+\infty$ and  $\mathrm{Cycl}(d,\alpha,\beta)_{\pmb{\pm}}=\{\Omega(\omega_1),\ldots, \Omega(\omega_{q})
	\}\not=\emptyset$. Then, the set $G_\infty(d,\alpha,\beta)_{\pmb{\pm}}$ may be empty set  or not ($(d,\alpha,\beta)_{\pmb{\pm}}$ may have  divergent trajectories). So, we have
	$$ \mathbb{N}=
	G_{\omega_1}(d,\alpha,\beta)_{\pmb{\pm}}
	\cup \ldots \cup G_{\omega_q}(d,\alpha,\beta)_{\pmb{\pm}}\cup G_\infty(d,\alpha,\beta)_{\pmb{\pm}}.$$
	In the case where $G_\infty(d,\alpha,\beta)_{\pmb{\pm}}$ is not empty, the subsets
	$G_{\omega_1}(d,\alpha,\beta)_{\pmb{\pm}}$,\ldots, $G_{\omega_q}(d,\alpha,\beta)_{\pmb{\pm}}$  and $G_\infty(d,\alpha,\beta)_{\pmb{\pm}}$  form a partition of $\mathbb{N}$.  Then,
	$\forall n\in\mathbb{N}$, $\exists k\in\mathbb{N}_0$ such that 
	$T^{(k)}(n)\in \{\omega_1,\ldots, \omega_{q},\infty\}$.

	\item If  $(d,\alpha,\beta)_{\pmb{\pm}}$  is a
	strongly admissible  triplet,  then 	 
$(d,\alpha,\beta)_{\pmb{\pm}}$
	has  a   finite  number $q=\# \mathrm{Cycl}(d,\alpha,\beta)_{\pmb{\pm}} $ of cycles denoted as  $\Omega(\omega_1),\ldots, \Omega(\omega_{q})$ with  $1\leq q<+\infty$ and  $\mathrm{Cycl}(d,\alpha,\beta)_{\pmb{\pm}}=\{\Omega(\omega_1),\ldots, \Omega(\omega_{q})
	\}\not=\emptyset$.
	In this case, the set $G_\infty(d,\alpha,\beta)$  is  an empty set and there is no divergent trajectory. 	Then, the subsets 	$G_{\omega_1}(d,\alpha,\beta),\ldots, G_{\omega_q}(d,\alpha,\beta)$   form a partition of $\mathbb{N}$:
	$$ \mathbb{N}=
	G_{\omega_1}(d,\alpha,\beta)
	\cup \ldots \cup G_{\omega_q}(d,\alpha,\beta),$$
and
	$\forall n\in\mathbb{N}$, $\exists k\in\mathbb{N}_0$ such that 
	$T^{(k)}(n)\in \{\omega_1,\ldots, \omega_{q}\}$.
\end{itemize}

The classical Collatz conjecture may  now be reformulated as following,
\begin{conjecture}\label{CollatzConjecture}
	The   triplet  $(2,3,1)_{\pmb{+}}$ is a  strongly  admissible   triplet of order $\# \mathrm{Cycl}(2,3,1)_{\pmb{+}}=1 $ .  Its unique cycle is the trivial one
	$\Omega(1)=(1\rightarrow 2\rightarrow 1)$ of length $2$.  For all $n\geq 1$, there exists an integer $k\geq 0$ such that $C^{(k)}(n)=1$, where the mapping $C$ is given by \eqref{CollatzMap}.
\end{conjecture}
Our conjecture~\ref{Myconjecture}  presented in the introduction 
may be also  reformulated as following:
\begin{conjecture}\label{MyConjecture_2}
	The   triplet  $(10,12,8)_{\pmb{+}}$ is a strongly admissible   triplet of order $\# \mathrm{Cycl}(10,12,8)_{\pmb{+}}=1$.  Its unique cycle is the trivial one
	$\Omega(4)=(4\rightarrow 8\rightarrow 16\rightarrow 24\rightarrow32\rightarrow40\rightarrow4)$ of length $6$. For all $n\geq 1$, there exists an integer $k\geq 0$ such that $T^{(k)}(n)=4$, where the mapping $T$ is given by \eqref{mapT0}.
\end{conjecture}

Determining whether a given triplet is admissible is not always straightforward. For instance, the triplets $(2,9,1)_{\pmb{+}}$,$(3,10,8)_{\pmb{+}}$, $(3,52,11)_{\pmb{+}}$, $(5,31,4)_{\pmb{+}}$, $(5,46,6)_{\pmb{-}}$ and $(15,35,10)_{\pmb{+}}$ appear to be non-admissible, although they satisfy condition  \eqref{CdtCollatztriplet}. For a given admissible triplet of integers, a major difficulty consists in determining whether it is weakly admissible, strongly admissible, or neither of the two. This question remains an open problem even in the classical case of the triplet  $(2,3,1)_{\pmb{+}}$. In this work, we formulate a general conjecture associated with specific families of admissible triplets, aiming to provide a broader framework for understanding the structural behavior of Collatz-type mappings. Let us now give  few examples of  admissible triplets $(d,\alpha,\beta)_{\pmb{\pm}}$. The assertions presented below are based on analytical observations, empirical evidence, and extensive computational experiments. The rationale behind these statements will be discussed in the next section; for the moment, the reader is invited to accept the following examples as illustrative. In the next section, we provide a detailed discussion of the underlying motivations and theoretical considerations supporting these statements. We also present several examples of admissible triples, together with their associated cycles.

\begin{examples}\label{Example201}
	
	\begin{enumerate}
		\item\label{Item-1.}  For a first example,  we consider the
	 triplet 
	$\bigl(33445533,33445534,33445532\bigr)_{\pmb{+}}$ with
		a large number $d =33445533$.
	 Its corresponding
		mapping  $T:\mathbb{N}\longrightarrow \mathbb{N} $ is given  by
		\begin{equation}\label{mapT_Exp1}
			T(n)=\left\{\begin{array}{ccc}
				n/33554433&\mathrm{if}& \modd{n}{0}{33554433}\\ 
				\Bigl(33554434 n+33554432 [ n]_d\Bigr)/33554433&\mathrm{if}& \notmodd{n}{0}{33554433}.\\
			\end{array}\right.
		\end{equation}
		It seams  that
		$\bigl(33445533,33445534,33445532\bigr)_{\pmb{+}}$
		is a strongly  admissible  triplet  of order one, $\# \mathrm{Cycl}\bigl(33445533,33445534,33445532\bigr)_{\pmb{+}}=1$. Its  trivial cycle of length $33554433$ is 
		$$
		\begin{array}{ll}
			\Omega(33554432)=
			(33554432 \rightarrow 67108864\rightarrow 100663296\rightarrow 134217728 \rightarrow 167772160
			\rightarrow 
			201326592 \rightarrow &\\
			234881024\rightarrow 
			268435456\rightarrow 
			301989888\rightarrow 
			335544320\rightarrow 
			369098752\rightarrow 
			402653184\rightarrow 
			\ldots\ldots\ldots\ldots &\\
			\ldots\ldots\rightarrow 1125899571298304\rightarrow 
			1125899604852736\rightarrow 
			1125899638407168\rightarrow 
			1125899671961600\rightarrow &\\
			1125899705516032\rightarrow 
			1125899739070464\rightarrow 
			1125899772624896\rightarrow 
			1125899806179328\rightarrow&\\
			1125899839733760\rightarrow
			1125899873288192\rightarrow
			1125899906842624\rightarrow 1125899940397056 \rightarrow 33554432).&
		\end{array}
		$$
		If  the fact that this triplet is a strong admissible triplet  with order one is true, 
		
			\item\label{Item-2}  In this example,
		we consider the  triplet $(5,6,373769)_{\pmb{+}}$, its corresponding 
		mapping  $T:\mathbb{N}\longrightarrow \mathbb{N} $ is given  by
		\begin{equation}\label{mapT_Exp5}
			T(n)=\left\{\begin{array}{ccc}
				n/5&\mathrm{if}& \modd{n}{0}{5}\\ 
				\Bigl(6 n+373769 [ n]_5\Bigr)/5&\mathrm{if}& \notmodd{n}{0}{5}.\\
			\end{array}\right.
		\end{equation}
the triplet  $(5,6,373769)_{\pmb{+}}$ has at least
$33$ cycles all of  same length $5$ which are:\\
$
\begin{array}{ll}
	\Omega(1331), 	\Omega(1936), 	\Omega(2057), 	\Omega(2541), 	\Omega(2662), 	\Omega(2783), 	\Omega(3146),  	\Omega(3267), 	\Omega(3388), 	\Omega(3509),	\Omega(3751), 	&\\
	\Omega(3872), 	\Omega(3993), 	\Omega(4114), 	\Omega(4477), 	\Omega(4598), 	\Omega(4719), 	\Omega(5203), 	\Omega(5324), 	\Omega(5929), 	\Omega(6776), 	\Omega(7502), 	&\\
	\Omega(8228), 	\Omega(8954), 	\Omega(9801), 	\Omega(10527), 	\Omega(11253), 	\Omega(11979), 	\Omega(12826), 	\Omega(13552), 	\Omega(14278), 	\Omega(15004), &\\	\Omega(373769).&\\
\end{array}
$
		
\item\label{Item-3}
		In this example, we consider the triplet $(4,10,54)_+$. The  corresponding 
		mapping  $T:\mathbb{N}\longrightarrow \mathbb{N} $ is given  by
		\begin{equation}\label{mapT_Exp(4-10-54)}
			T(n)=\left\{\begin{array}{ccc}
				n/4&\mathrm{if}& \modd{n}{0}{4}\\ 
				\Bigl(10 n+54 [ n]_4\Bigr)/4=	\Bigl(5 n+27 [ n]_4\Bigr)/2&\mathrm{if}& \notmodd{n}{0}{4},\\
			\end{array}\right.
\end{equation}			
It seems that  $(4,10,54)_+$ is a strongly admissible triplet of order  $\# \mathrm{Cycl}(4,10,54)_+=  37$. Its trivial $37$ cycles are stored in increasing length as follows:
		
\begin{tabular}{llll}
- Cycles of length $2$ :&  $\Omega(9)$,   $\Omega(18)$,  $\Omega(27)$&
- Cycle of length $20$ :& $\Omega(189)$,\\
-  Cycles of length $3$ : &$\Omega(1)$,  $\Omega(2)$, $\Omega(3)$,&
- Cycle of length $25$ : &	$\Omega(342)$,\\
- Cycles of length $5$ : &
$\Omega(477)$, $\Omega(549)$, $\Omega(693)$, &
- Cycles of length $27$ : &	$\Omega(78)$,  $\Omega(93)$,\\
&$\Omega(702)$,
$\Omega(774)$,
$\Omega(837)$,&
- Cycles of length $30$ : &	$\Omega(198)$,  $\Omega(237)$,\\		
&$\Omega(918)$,
$\Omega(927)$,
$\Omega(999)$,
 &
- Cycle of length $36$ : &
$\Omega(13)$,\\
&
$\Omega(1062)$,
$\Omega(1143)$,
$\Omega(1287)$,	&- Cycles of length $98$ : &	$\Omega(5967)$,\\
- Cycle of length $6$ : &	$\Omega(6)$,&
- Cycle of length $108$ : & $\Omega(1518)$,\\
- Cycle of length $10$ :&  $\Omega(639)$,&
- Cycle of length $246$ : & $\Omega(214)$,\\
- Cycles of length $15$ : & $\Omega(7)$, $\Omega(678)$,&\\ 
- Cycle of length $583$ : & $\Omega(25983)$, $\Omega(31662)$, $\Omega(33534)$,&
- Cycle of length $681$ : & $\Omega(4174)$, $\Omega(14927)$.\\ 
\end{tabular}

\item\label{Item-3}
In this example, we consider the triplet $(3,4,1)_-$ its  corresponding 
mapping  $T:\mathbb{N}\longrightarrow \mathbb{N} $ is given  by
\begin{equation}\label{mapT_Exp(3-4-1)_}
	T(n)=\left\{\begin{array}{ccc}
		n/3&\mathrm{if}& \modd{n}{0}{3}\\ 
		\Bigl(4 n+ [ -n]_3\Bigr)/3&\mathrm{if}& \notmodd{n}{0}{3},\\
	\end{array}\right.
\end{equation}			
It seems that $(3,4,1)_-$ is a strong admissible triplet with
$\# \mathrm{Cycl}(3,4,1)_-= 2$ its trivial cycles are:
$\Omega(1)=(1\rightarrow 2\rightarrow3\rightarrow1)$ of length $2$   and 
$\Omega(7)=(7\rightarrow 10\rightarrow 14\rightarrow 19\rightarrow 26\rightarrow 35\rightarrow 47\rightarrow 63\rightarrow 21\rightarrow 7)$ of length $9$.	

\item\label{item-3A}  In this example,
we consider the triplet  $(5,6,4)_{\pmb{+}}$, corresponding to 
$d = 5$, $\alpha =  6$, $\beta = 4$ and $\kappa_0=1$. Its corresponding 
mapping  $T:\mathbb{N}\longrightarrow \mathbb{N} $ is given  by
\begin{equation}\label{mapT_Exp5}
	T(n)=\left\{\begin{array}{ccc}
		n/5&\mathrm{if}& \modd{n}{0}{5}\\ 
		\Bigl(6 n+4[ n]_5\Bigr)/5&\mathrm{if}& \notmodd{n}{0}{5}.\\
	\end{array}\right.
\end{equation}
It seems  that  the 
triplet  $(5,6,4)_{\pmb{+}}$
is a strongly admissible  triplet of order one.  Its  trivial  cycle is
$\Omega(4)=(4\rightarrow 8\rightarrow 12 \rightarrow 16\rightarrow 20\rightarrow 4)$ of length $5$. 
\end{enumerate}
\end{examples}

\section{Special classes of  admissible  triplets and conjectures}

In this section, we will give  some  classes of admissible triplets $(d,\alpha,\beta)_{\pmb{\pm}}$.  
The particular  case for  $d=2$,   has been studied, by the author,   see \cite{Bouhamidi-Syracuse2021} for more details.  We will also formulate some general conjectures.

\subsection{The   admissible triplets  $(d, \alpha,\beta)_{\pmb{\pm}}$ with $\beta =\kappa_0(d^{\nu_0}-\alpha)$,
	$ \alpha=d^{\nu_1}-\kappa_1\delta$,  $ 1\leq \delta\leq d-1$, $\nu_0,\nu_1\geq 1$ and  $\kappa_0, \kappa_1=\pm 1$. }\label{subsubsection:case1}

We have the following theorem.
\begin{theorem}\label{trivialcycleTheoremA_OLD}
		Let  $d\geq 2$, let $\nu_0, \nu_1\geq 1$, $\delta\in\{1,\ldots,d-1\}$ and  $\kappa_0,\kappa_1\in\{\pm 1\}$. Define 
	$\alpha =d^{\nu_1}-\kappa_1\delta $ and 
	$\beta = \kappa_0(d^{\nu_0}-\alpha)$.  Then the triplet $(d,\alpha,\beta)_{\pmb{\pm}}$ is admissible in each the following cases. 
	\begin{enumerate}
		\item  {\bf The case }  $\kappa_0=+1$. In this case,
		$(d,\alpha,\beta)_{\pmb{+}}$ admits at least
		$d-1$   trivial   cycles of length $\nu_0$ given, for each
		$r=1,\ldots, d-1$,  by:
		\begin{equation}\label{TrivialCycles1}
			\Omega(r)=	( r\rightarrow rd^{\nu_0-1}\rightarrow r d^{\nu_0-2}\rightarrow\ldots\rightarrow rd \rightarrow  r). 
		\end{equation}
		
		\item {\bf The case}   $\kappa_1\beta>0$. In this case, the triplet  $(d,\alpha,\beta)_{\pmb{\pm}}$ admits trivial cycles of length $\nu_1$, described as follows
		\begin{enumerate}
			\item[2-1.]  If $\delta=1$,  then   the  $d-1$  trivial cycles, for each $r=1,\ldots, d-1$, are
			\begin{equation}\label{TrivialCycles2}
				\Omega(r|\beta|)=	( r |\beta|\rightarrow r |\beta| d^{\nu_1-1}\rightarrow r|\beta|  d^{\nu_1-2}\rightarrow\ldots\rightarrow r|\beta| d\rightarrow  r |\beta|)
			\end{equation}
			Moreover, if $\kappa_0=+1$ and $\nu_0=\nu_1$,  these  cycles  in  \eqref{TrivialCycles2} coincide with those in \eqref{TrivialCycles1}.
			\item[2-2.] If $1< \delta\leq d-1$,  $\nu_0$, $\nu_1\geq 2$ and  $\nu_0\not=\nu_1$, let $q_0=gcd(d^{\nu_0-1}-d^{\nu_1-1},\delta)$.  Then there exist two integers $\beta_0$ and $\delta_0$ such that   $\beta=q_0\beta_0$, $\delta=q_0\delta_0$  and  the triplet
			$(d,\alpha,\beta)_{\pmb{\pm}}$ admits  the  trivial cycles $\Omega(r|\beta_0|)$ of length $\nu_1$,  described as:
			\begin{equation}\label{TrivialCycles3}
				\Omega(r|\beta_0|)=	( r|\beta_0|\rightarrow r|\beta_0| d^{\nu_1-1}\rightarrow r|\beta_0|  d^{\nu_1-2}\rightarrow\ldots\rightarrow r|\beta_0| d\rightarrow  r|\beta_0|),
			\end{equation}
			for each $r=1,\ldots,\left\lfloor\dfrac{d-1}{\delta_0}\right\rfloor$, where $\lfloor\;\;\rfloor$ denotes the floor function.		
		\end{enumerate}
	\end{enumerate}
\end{theorem}

\begin{proof}
	First, it is easy to check that the condition \eqref{CdtCollatztriplet} is satisfied.
	\begin{enumerate}
		\item For  $\kappa_0=1$ and  for all  $r\in\{1,\ldots,d-1\}$,  we have
		$T(r)=r\,\dfrac{\alpha+\beta}{d}=r\,\dfrac{d^{\nu_0}}{d}=r\,d^{\nu_0-1}$.
		It follows that $T^{(\nu_0)}(r)=r$, which gives rise to the trivial cycles given by  \eqref{TrivialCycles1}. Thus, we have   
		$\# \mathrm{Cycl}(d,\alpha,\beta)_+\geq d-1$. 	
		\item   For   $\kappa_0=\pm 1$, if  the condition $\kappa_1\beta>0$ holds, then:
		\begin{enumerate}
			\item[2-1.] If $\delta=1$, we have  $\alpha=d^{\nu_1}-\kappa_1$  and
			$\beta=\kappa_0(d^{\nu_0}-\alpha)$. As   $\kappa_1\beta>0$, then, for  all  $r\in \{1,\ldots,d-1\}$,  we have
			$\kappa_1 r\beta>0$,
			$\kappa_0\kappa_1 r\beta=\kappa_1r(d^{\nu_0}-d^{\nu_1})+r$ and $[\kappa_0\kappa_1 r\beta]_d=r$. It follows that,
			$$T(\kappa_1r\beta )=\dfrac{\alpha\kappa_1r \beta+\beta[\kappa_0\kappa_1r\beta]_d}{d}=\beta r\,\dfrac{\kappa_1\alpha+1}{d}.$$
			But, $\kappa_1\alpha=\kappa_1d^{\nu_1}- 1 $, then 
			$T(\kappa_1r\beta )=\kappa_1r\beta\,\dfrac{d^{\nu_1}}{d}=\kappa_1r\beta\,d^{\nu_1-1}$.
			It follows that $T^{(\nu_1)}(\kappa_1r\beta)=\kappa_1r\beta$. As $|\beta|=\kappa_1\beta$, this gives rise to  the trivial cycles in  (\ref{TrivialCycles2}) and  $\# \mathrm{Cycl}(d,\alpha,\beta)_{\pmb{-}}\geq d-1$.   But for $\kappa_0=1$: 
			If $\nu_0=\nu_1$, then  $\kappa_1\beta=1$ and
			the cycles in (\ref{TrivialCycles1})  and (\ref{TrivialCycles2}) coincide.  Thus,  $\# \mathrm{Cycl}(d,\alpha,\beta)_{\pmb{+}}\geq d-1$, else  if $\nu_0\not=\nu_1$, then  the cycles in (\ref{TrivialCycles1})  and (\ref{TrivialCycles2}) are different and  
			$\# \mathrm{Cycl}(d,\alpha,\beta)_{\pmb{+}}\geq (d-1)+(d-1)=2(d-1)$.
			\item[2-2.]  If $\;1< \delta\leq d-1$ and   $\nu_0$, $\nu_1$ are two integers  $\geq 2$ with $\nu_0\not=\nu_1$. \\Let  $q_0=gcd(d^{\nu_0-1}-d^{\nu_1-1},\delta)>0$, then there exist two integers $q_1$ and $\delta_0>0$ such that  $d^{\nu_0-1}-d^{\nu_1-1}=q_1q_0$ and  $\delta=q_0\delta_0$. Thus, $\beta=q_0\beta_0$ with $\beta_0=\kappa_0(dq_1+\kappa_1\delta_0)$. It follows that, 
			for all integer $r\in\{1,\ldots,d-1\}$, we have 	$\kappa_1 r\beta_0=\kappa_0(\kappa_1rq_1d+r\delta_0)>0$ since $\kappa_1\beta=q_0\kappa_1\beta_0>0$ and $q_0>0$. For $r$ such that
			$r \leq \left\lfloor\dfrac{d-1}{\delta_0}\right\rfloor$, we have  $0<r\delta_0 \leq d-1$ and
			$[\kappa_0\kappa_1 r\beta_0]_d=r\delta_0$.   Then 
			\begin{eqnarray*}
				T(\kappa_1r\beta_0 )=\dfrac{\alpha\kappa_1 r\beta_0+\beta[\kappa_0\kappa_1r\beta_0]_d}{d}=\dfrac{\alpha\kappa_1 r\beta_0+r\delta_0\beta }{d}=\dfrac{\alpha\kappa_1 r\beta_0+r\delta_0q_0\beta_0 }{d}=r\beta_0\,\dfrac{\kappa_1\alpha+\delta}{d}.
			\end{eqnarray*}
			But, $\kappa_1\alpha=\kappa_1d^{\nu_1}- \delta$, we get
			$T(\kappa_1r\beta_0 )=\kappa_1r\beta_0\,\dfrac{d^{\nu_1}}{d}=\kappa_1r\beta_0\,d^{\nu_1-1}$.
			As $|\beta_0|=\kappa_1\beta_0$, this gives rise to the trivial cycles given by (\ref{TrivialCycles3}) 
			and   then 
			$\# \mathrm{Cycl}(d,\alpha,\beta)_{\pmb{+}}\geq (d-1)+\left\lfloor\dfrac{d-1}{\delta_0}\right\rfloor$ and
			$\# \mathrm{Cycl}(d,\alpha,\beta)_{\pmb{-}}\geq \left\lfloor\dfrac{d-1}{\delta_0}\right\rfloor$.		
		\end{enumerate}
	\end{enumerate}
	\hfill\end{proof}

\begin{remarks}
	For the classical Collatz  triplet  $(1,2,3)_{\pmb{+}}$, we have $d=2$,   $\alpha=3$, $\beta=1$,  $\kappa_0=+1$ and $\alpha+\beta = 4=2^{\nu_0} $, with $\nu_0=2$. In this case, 
	we may write $\alpha$ and $\beta$ in the two following forms:
	\begin{enumerate}
		\item   $\alpha =3=2^{\nu_1}-\kappa_1\delta$ and $\beta =1= 2^{\nu_0}-\alpha$ with $\nu_1=\nu_0=2$ and $\kappa_0=\kappa_1=\delta=1$. 
		\item   $\alpha =3=2^{\nu_1}-\kappa_1\delta$ and $\beta =1= 2^{\nu_0}-\alpha$ with $\nu_1=1$,  $\nu_0=2$, $\delta=\kappa_0=1$ and $\kappa_1=-1$. In this case, $\kappa_1\beta<0$.  
	\end{enumerate}
	According to the previous theorem,  for  both  last cases, we should have
	at least one trivial cycle  of length $\nu_0=2$:  $\Omega(1)=(1\rightarrow 2\rightarrow1)$.
\end{remarks}

Let us now give few examples that  illustrate the previous Theorem.
\begin{examples}\label{Examples3.1}
	\begin{enumerate}
		\item 	Let us choose for instance the parameters
		$d=3$,  $\kappa_0=\kappa_1=\delta=1$, $\nu_0=3$ and $\nu_1=2$. Then $\alpha = d^{\nu_1}-\kappa_1\delta=8$ and $\beta =\kappa_0(d^{\nu_0}-\alpha)=19$. The  corresponding 
		mapping  $T:\mathbb{N}\longrightarrow \mathbb{N} $ is given  by
		\begin{equation}\label{mapT_Exp2}
			T(n)=\left\{\begin{array}{ccc}
				n/3&\mathrm{if}& \modd{n}{0}{3}\\ 
				\Bigl(8 n+19 [ n]_3\Bigr)/3&\mathrm{if}& \notmodd{n}{0}{3}.\\
			\end{array}\right.
		\end{equation}
This corresponds to  Item~2-1 in the previous Theorem and we have $\kappa_1\beta>0$, then
 $(3,8,19)_+$ is  an  admissible triplet having  at least two cycles 
$\Omega(1)$  and $\Omega(2)$  of length $\nu_0=3$ and the two others cycles 
$\Omega(19)$ and  	$\Omega(2\times 19=38)$ of length $\nu_1=2$. We have
		$
		\Omega(1)=(1\rightarrow 9\rightarrow 3\rightarrow  1)$, 
		$\Omega(2)=(2\rightarrow18\rightarrow 6\rightarrow  2)$, 
		$\Omega(19)=(19\rightarrow 57\rightarrow 19)$  and 
		$\Omega(38)=(38\rightarrow 114\rightarrow 38)$.
		For this example we may have divergent trajectories.
	It seems   that  $(3,8,19)_{\pmb{+}}$  is a weakly  admissible  triplet with order  $\# \mathrm{Cycl}(3,8,19)_{\pmb{+}}= 4$.

		\item In this example, we choose, for instance, the parameters: $d=3$, $\kappa_0=\kappa_1=\delta=1$, $\nu_0=2$ and $\nu_1=3$. Then $\alpha = d^{\nu_1}-\kappa_1\delta=26$ and $\beta =\kappa_0(d^{\nu_0}-\alpha)=-17$.  The corresponding mapping  to the   triplet $(3,26,-17)_{\pmb{+}}$ is  given as 
		$$
		T(n)=\left\{\begin{array}{ccc}
			n/3&\mathrm{if}& \modd{n}{0}{3}\\ 
			\Bigl(26 n-17 [ n]_3\Bigr)/3&\mathrm{if}& \notmodd{n}{0}{3}.\\
		\end{array}\right.$$
		As here   $\kappa_1\beta<0$, this example  corresponds to the item~1, in the previous Theorem. Then $(3,26,-17)_{\pmb{+}}$ is an admissible triplet of order  $\# \mathrm{Cycl}(3,26,-17)_{\pmb{+}}\geq d-1=2$. In this case, we have at least $2$
		trivial  cycles of length $\nu_0=2$ which are: $\Omega(1)=(1\rightarrow 3\rightarrow 1)$ and 
		$\Omega(2)=(2\rightarrow 6\rightarrow 2)$.
		\item In this example, we take the similar parameters as in the previous example, we just change  $\kappa_1=-1$. Then 
		$\alpha = d^{\nu_1}-\kappa_1\delta=28$ and $\beta =\kappa_0(d^{\nu_0}-\alpha)=-19$.  The corresponding mapping to the triplet $(3,28,-19)_{\pmb{+}}$ is  given as 
		$$
		T(n)=\left\{\begin{array}{ccc}
			n/3&\mathrm{if}& \modd{n}{0}{3}\\ 
			\Bigl(28 n-19 [ n]_3\Bigr)/3&\mathrm{if}& \notmodd{n}{0}{3}.\\
		\end{array}\right.
		$$
		As in this case   $\kappa_1 \beta>0$, $\delta=1$ and $\nu_0\not=\nu_1$,  this  example  corresponds to the item~2-1  in the previous Theorem.  Then $(3,28,-19)_{\pmb{+}}$ is an admissible triplet of  order $\# \mathrm{Cycl}(3,28,-19)_{\pmb{+}}\geq 2(d-1)=4$. In this case, we have at least $4$ trivial  cycles. Two of  length $\nu_0=2$ which are: $\Omega(1)=(1\rightarrow 3\rightarrow 1)$ and 
		$\Omega(2)=(2\rightarrow 6\rightarrow 2)$ and  two others of length $\nu_1=3$, which are: $\Omega(|\beta|=19)=(19\rightarrow 171\rightarrow 57\rightarrow 19)$ and 
		$\Omega(2|\beta|=38)=(38\rightarrow  342\rightarrow 114\rightarrow 38)$.
		
		\item In this example,  we consider the   triplet 
		$(12,134,1594)_{\pmb{+}}$
		corresponding to the parameters 
		$d=12$, $\nu_0=3$, $\nu_1=2$, $\delta=10$ and $\kappa_0=\kappa_1=1$. Then $\alpha=d^{\nu_1}-\kappa_1\delta=134$ and $\beta=\kappa_0(d^{\nu_0}-\alpha)=1594$. The  corresponding mapping is:
		$$
		T(n)=\left\{\begin{array}{ccc}
			n/12&\mathrm{if}& \modd{n}{0}{12}\\ 
			\Bigl(134n+1594 [ n]_{12}\Bigr)/12&\mathrm{if}& \notmodd{n}{0}{12}.\\
		\end{array}\right.
		$$
		We have $\kappa_1\beta >0$ and $\nu_0\not=\nu_1$ with $\delta = 10>1$. This corresponds to Item~2-2 in the previous Theorem.  Then $(12,134,1594)_{\pmb{+}}$  is an  admissible triplet of 
		$\# \mathrm{Cycl}(12,134,1594)_{\pmb{+}}\geq d-1=11$. It follows that  $(12,134,1594)_{\pmb{+}}$ has at least
		$d-1=11$ trivial cycles  of length $\nu_0=3$, which are:
		$\Omega(r)=(r\rightarrow 144r\rightarrow 12r\rightarrow r)$  for $1\leq r\leq  11$.
		As  $q_0= gcd(d^{\nu_0-1}-d^{\nu_1-1},\delta)=2$.  
		Then $\delta_0=5$, $\beta_0= 797$ and 
		$\left\lfloor\dfrac{d-1}{\delta_0}\right\rfloor= 2$. It follows that  the two  other cycles  of length $\nu_1=2$ are:
		$\Omega(797r)= (797r\rightarrow 9564 r\rightarrow 797r)$ for $r=1,2$.
		Finally, the admissible  triplet $(12,134,1594)_{\pmb{+}}$ has at least  $13$ cycles.
		
		\item  In this example, we consider the   triplet 
		$(3,8,5)_{\pmb{-}}$ corresponding to the parameters:  $d=3$, $\nu_0=1$, $\nu_1=2$, $\delta=1$,  $\kappa_0=-1$ and $\kappa_1=1$.
		Then $\alpha=d^{\nu_1}-\kappa_1\delta=8$ and $\beta=\kappa_0(d^{\nu_0}-\alpha)=5$. The  associated  mapping is:
		$$
		T(n)=\left\{\begin{array}{ccc}
			n/3&\mathrm{if}& \modd{n}{0}{3}\\ 
			\Bigl(8n+5 [ -n]_{3}\Bigr)/3&\mathrm{if}& \notmodd{n}{0}{3},\\
		\end{array}\right.\;\textrm{or}\;
T(n)=\left\{\begin{array}{ccc}
			n/3&\mathrm{if}& \modd{n}{0}{3}\\ 
			\Bigl(8n+10 \Bigr)/3&\mathrm{if}& \modd{n}{1}{3},\\
			\Bigl(8n+5\Bigr)/3&\mathrm{if}& \modd{n}{2}{3}.\\
		\end{array}\right.
		$$
		We have $\kappa_1\beta>0$  and $\delta=1$. This corresponds to Item~2-1 in the previous Theorem.  Then $(3,8,5)_{\pmb{-}}$ is an admissible triplet having at least
$d-1=2$ trivial cycles  of length $\nu_1=2$, which are:
		$\Omega(r|\beta|)=(5r\rightarrow15r\rightarrow  5r)$  for $1\leq r\leq  2$. Namely, $\Omega(5)=(5\rightarrow15\rightarrow  5)$  and $\Omega(10)=(10\rightarrow30\rightarrow  10)$ .

		\item Now let us give a general example: For   $d\geq 2$, we choose $\delta= 1$,  $\kappa_1=-1$ and $\kappa_0=1$ with  
		$\nu_1>\nu_0\geq 2$,  so $\kappa_1 \beta=-\beta>0$. We get
		$\alpha=d^{\nu_1}+1$ and $\beta=d^{\nu_0}-\alpha=d^{\nu_0}-d^{\nu_1}-1$. Then  the triplet $(d,\alpha,\beta)_{\pmb{+}}$ has at least  $2(d-1)$ trivial cycles  which are
		are:  
		$$
		\Omega(r)=	( r\rightarrow rd^{\nu_0-1}\rightarrow r d^{\nu_0-2}\rightarrow\ldots\rightarrow rd \rightarrow  r),
		$$
		and
		$$
		\Omega(r|\beta|)=	( r |\beta|\rightarrow r |\beta| d^{\nu_1-1}\rightarrow r|\beta|  d^{\nu_1-2}\rightarrow\ldots\rightarrow r|\beta| d\rightarrow  r |\beta|),
		$$
		of length $\nu_0$ and $\nu_1$, respectively, for all $r\in\{1,\ldots,d-1\}$.
	\end{enumerate}
	
\end{examples}

\subsection{\textbf{The admissible triplets $ (d,\alpha,\beta)_{\pmb{+}}$ with  $\alpha = d^{\nu_1}+1$ and $\beta=d^{2\mu_0+\nu_1}-\alpha^2$ for  $ d\geq 2$ with $\mu_0\geq 1$ and $2\mu_0>\nu_1\geq 1$.}}
Let us return back to the classical Collatz case. The   $3x+k$ problem were previously studied in \cite{Lagarias1990} it corresponds to the triplets $(2,3,k)_{\pmb{+}}$ problem. In \cite{Matthews2010},  it is indicated that in
2002,  A. S. Jones pointed out  that  for the problem $3x+k$,  if $k= 2^{2c+1}-9$, the numbers $2^k+3$ for $1\leq k\leq c$, generate $c$    different cycles of length $2c+1$. We aim in this subsection,  to give an extension to  the result pointed out by A. S. Jones.  Let $d\geq 2$ and 
consider the particular  triplet  $(d,\alpha,\beta)_{\pmb{+}}$  where 
 $\alpha = d^{\nu_1}+1$ and $\beta=d^{2\mu_0+\nu_1}-\alpha^2$ for  $ d\geq 2$ with $\mu_0\geq 1$ and $2\mu_0>\nu_1\geq 1$.  Then $\alpha+\kappa_0\beta =\alpha + d^{2\mu_0+\nu_1}-\alpha^2=\lambda_0d^{\nu_1}$  with $\lambda_0= d^{2\mu_0}-d^{\nu_1}-1$ and $\kappa_0=1$. It follows that the triplet $ (d,\alpha,\beta)_{\pmb{+}}$  satisfies the condition \eqref{CdtCollatztriplet} and $\beta = \lambda_0d^{\nu_1}-\alpha$.
We may state  the following result.
\begin{theorem}\label{trivialcycle3.4} 
	Let $d\geq 2$,  $\mu_0\geq 1$ and $\nu_1\geq 1$ such that $2\mu_0> \nu_1$ and consider the  triplet $(d,\alpha,\beta)_{\pmb{+}}$ with 
	$\alpha =d^{\nu_1}+1$ and $\beta = d^{2\mu_0+\nu_1}-\alpha^2$. Then, $(d,\alpha,\beta)_{\pmb{+}}$  is an admissible triplet with  at least $q$ cycles  
	$\Omega(\omega_{k,r})$,   of length $2\mu_0+\nu_1$    with $q\geq  \mu_0(d-1)^2$, generated by the starting numbers
	$\omega_{k,r}=r_1(r_2d^k+r_3\alpha)$ 
	for all $k$ and all  $r=(r_1,r_2,r_3)\in\mathbb{N}^3$
	such that
	\begin{equation}\label{kr1r2r3}
		1\leq k\leq \mu_0,\;
		1\leq r_1\leq d-1, \;
		1\leq r_2,\leq d-1,\;
		1\leq  r_3\leq d-1,\;
		1\leq r_1r_2\leq d-1, \;
		1\leq r_1r_3\leq d-1.
	\end{equation}
	Furthermore:
	\begin{itemize}
		\item  If $\nu_1=1$, then   $(d,\alpha,\beta)_{\pmb{+}}$  has the further cycle $\Omega(\beta)$ of length $d$,
		\item  If  $\nu_1>1$, then   the sequence $(s(n))_{n\geq 1}$, where $\displaystyle s(n):=\dfrac{1}{n}\sup_{\ell\geq 1}T^{(\ell)}(n)$,
		is unbounded.
	\end{itemize}
\end{theorem}
\begin{proof}
	For $r=(r_1,r_2,r_3)$ satisfying \eqref{kr1r2r3},  we have  $[\omega_{k,r}]_d=[r_1r_3\alpha]_d=r_1r_3$, then  for $k\geq 1$:
	$$T(\omega_{k,r})=\dfrac{\alpha \omega_{k,r}+\beta r_1r_3 }{d}=\dfrac{\alpha r_1(r_2d^k+r_3\alpha)+\beta r_1r_3 }{d}=
	\alpha r_1r_2d^{k-1}+r_1r_3\dfrac{\beta+\alpha^2}{d}.$$
	As $\alpha^2+\beta =d^{2\mu_0+\nu_1}$, it follows that
	$\displaystyle T(\omega_{k,r})=
	\alpha r_1r_2d^{k-1}+r_1r_3d^{2\mu_0+\nu_1-1}$. Then,  for $1\leq k\leq \mu_0$,  at the step $k$, we have
	$$T^{(k)}(\omega_{k,r})= \alpha r_1r_2+r_1r_3d^{2\mu_0+\nu_1-k}.$$
	As $[r_1r_2\alpha]_d=r_1r_2$, it follows that
	\begin{eqnarray}T^{(k+1)}(\omega_{k,r})=
		\dfrac{\alpha r_1(\alpha r_2+r_3d^{2\mu_0+\nu_1-k})+\beta r_1r_2 }{d}
		=	\alpha r_1r_3d^{2\mu_0+\nu_1-k-1}+r_1r_2\dfrac{\beta+\alpha^2}{d}\nonumber\\=\alpha r_1r_3d^{2\mu_0+\nu_1-k-1}+r_1r_2d^{2\mu_0+\nu_1-1}.\nonumber
	\end{eqnarray}
	Then,  by induction for $1\leq \ell\leq 2\mu_0+\nu_1-k$, we have 
	$\displaystyle T^{(k+\ell)}(\omega_{k,r})=\alpha r_1r_3d^{2\mu_0+\nu_1-k-\ell}+r_1r_2d^{2\mu_0+\nu_1-\ell}$, which gives  for $\ell=2\mu_0+\nu_1-k$ that
	$\displaystyle  T^{(2\mu_0+\nu_1)}(\omega_{k,r})=\omega_{k,r}$. I follows that  $\Omega(\omega_{k,r})$ is a cycle of length $2\mu_0+\nu_1$ for all $k$ and $r=(r_1,r_2,r_3) $ satisfying  \eqref{kr1r2r3} .
	 Furthermore: 
	 \begin{itemize}
	 	\item  	 If $\nu_1=1$, we have $\beta = d^{2\mu_0+1}-\alpha^2=d^{2\mu_0+1}-(d+1)^2$, it follows that $\alpha-1=d$,  $[\beta]_d=d-1$ and 
	 	$$T(\beta)=\dfrac{\alpha \beta+\beta[\beta]_d}{d}=\beta \dfrac{\alpha+d-1}{d}=2\beta.$$
	 	By induction, we have for $1\leq k\leq d-1$, $[k\beta]_d =d-k$ and 
	 	$$T^{(k)}(\beta)=\dfrac{k\alpha \beta+\beta[k\beta]_d}{d}=\beta \dfrac{k\alpha+d-k}{d}=(k+1)\beta.$$
	 	Then, $T^{(d)}(\beta)=\beta$, and $\Omega(\beta)$ is a cycle of length $d$.
	 	\item  If $\nu_1>1$, let us consider the sequence $n_k=\beta(d^k+1)$ for $k\geq 0$.   As $\alpha=d^{\nu_1}+1$, then
	 	$[n_k]_d = [\beta]_d=[-\alpha]_d=d-1$. It follows that 
	 	$$T(n_k)=\dfrac{\alpha \beta(d^k+1)+\beta(d-1)}{d}=\beta\dfrac{\alpha d^k+(\alpha-1)+d}{d}=\beta(\alpha d^{k-1}+d^{\nu_1-1}+1).$$
	 So, also
	 		 	$$T^{(2)}(n_k)=\dfrac{\alpha \beta(\alpha d^{k-1}+d^{\nu_1-1}+1)+\beta(d-1)}{d}=\beta(\alpha^2d^{k-2}+\alpha d^{\nu_1-2}+d^{\nu_1-1}+1).$$
It follows by induction that for $1\leq \ell\leq k$, we have	 	
$$T^{(\ell)}(n_k)=\beta\bigl(\alpha^{\ell}d^{k-\ell}+\sum_{i=0}^{\ell-1}\alpha^ id^{\nu_1-i-1}+1
\bigr).$$
Then, for  $1\leq \ell\leq k$, we have the following formula
$$T^{(\ell)}(n_k)=\beta\Bigl(
\bigl(\dfrac{\alpha}{d}\bigr)^{\ell}d^k
+d^{\nu_1}\dfrac{\bigl(\dfrac{\alpha}{d}\bigr)^{\ell}-1}{\alpha-d}+1\Bigl).$$
As $\alpha>d$, it follows that
$$\sup_{k\geq \ell\geq 1}T^{(\ell)}(n_k)=
\beta\Bigl(
\bigl(\dfrac{\alpha}{d}\bigr)^{k}d^k
+d^{\nu_1}\dfrac{\bigl(\dfrac{\alpha}{d}\bigr)^{k}-1}{\alpha-d}+1\Bigl),$$
and then
$$\dfrac{1}{n_k}\sup_{\ell\geq 1}T^{(\ell)}(n_k)\geq \dfrac{1}{n_k}\sup_{k\geq \ell\geq 1}T^{(\ell)}(n_k)=
\dfrac{d^k}{d^k+1}\Bigl(\dfrac{\alpha}{d}\Bigr)^{k}+
\dfrac{d^{\nu_1}}{d^k+1}\dfrac{\Bigl(\dfrac{\alpha}{d}\Bigr)^{k}-1}{\alpha-d}+\dfrac{1}{d^k+1}\geq 
\dfrac{d^k}{d^k+1}\Bigl(\dfrac{\alpha}{d}\Bigr)^{k}.$$
As $\alpha>d$, it follows that  $\displaystyle \dfrac{1}{n_k}\sup_{\ell\geq 1}T^{(\ell)}(n_k)\longrightarrow +\infty$ as $k\longrightarrow +\infty$.

\end{itemize}

	\hfill\end{proof}

\begin{example}\label{Exp(5,6,3089)+}
	\begin{itemize}
		\item For instance,  we choose  $d=5$, $\nu_1=1$ and $\mu_0=2$. It follows that $\alpha =d^{\nu_1}+1=6$ and  $\beta =d^{2\mu_0+\nu_1}-\alpha^2= 3089$.  Then, $(5,6,3089)_{\pmb{+}}$ is an admissible triplet 
		possessing at least the following $32$ cycles  of length $2\mu_0+\nu_1=5$
		generated by the starting numbers $\omega_{k,r}$ with 
		$k$ and $r=(r_1,r_2,r_3)$ satisfying the conditions \eqref{kr1r2r3}. The $32$ cycles are:
		$\Omega(11)$, 
		$\Omega(16)$,
		$\Omega(17)$,
		$\Omega(21)$,
		$\Omega(22)$,
		$\Omega(23)$,
		$\Omega(26)$,
		$\Omega(27)$,
		$\Omega(28)$,
		$\Omega(29)$,
		$\Omega(31)$,
		$\Omega(32)$,
		$\Omega(33)$,
		$\Omega(34)$,
		$\Omega(37)$,
		$\Omega(38)$,
		$\Omega(39)$,
		$\Omega(43)$,
		$\Omega(44)$,
		$\Omega(49)$,
		$\Omega(56)$,
		$\Omega(62)$,
		$\Omega(68)$,
		$\Omega(74)$,
		$\Omega(81)$,
		$\Omega(87)$,
		$\Omega(93)$,
		$\Omega(99)$,
		$\Omega(106)$,
		$\Omega(112)$,
		$\Omega(118)$ and 
		$\Omega(124)$.  Furthermore,  as in this case $\nu_0=1$, the triplet $(5,6,3089)_{+}$ has the additional   cycle
		$\Omega(\beta =3089)$ of length $d=5$ starting by $\beta=3089$. It appears that the triplet $(5,6,3089)_{\pmb{+}}$ is a strongly  admissible triplet with exactly the order $=33$.  
	\end{itemize}

\end{example}

\subsection{\textbf{The triplets $\bf  (d,\alpha,\beta_0)_{\pmb{\pm}}$
		and $\bf  (d,\alpha,\beta)_{\pmb{\pm}}$ with $\beta= a_0\beta_0$ and $\modd{a_0}{1}{d}$.}}

In this subsection, we will establish an extension  to a result   given in  \cite{Lagarias1990} for the classical Collatz case. More precisely, in  \cite{Lagarias1990},  Lagarias proposed the map $L_\beta$,  see Section~1.  We may observe that the trivial cycle   $\Omega(\beta)=(\beta \rightarrow 2\beta\rightarrow \beta)$  of the mapping $L_\beta$ is obtained by multiplying the trivial cycle
$\Omega(1)=(1  \rightarrow 2\rightarrow  1)$  of $L_1$  by $\beta$.  In the following theorem, we will give a general  similar  result for  an admissible  triplet  $(d,\alpha,\beta_0)_{\pmb{\pm}}$. 

\begin{theorem}
	For all $a_0$ positive  integer such  that
	$\modd{a_0}{1}{d}$, if
	$\Omega_0(\omega)$ is a cycle  of  length $\ell$  for  the admissible triplet $(d,\alpha,\beta_0)_{\pmb{\pm}}$   then $\Omega(a_0\omega)$ is also a cycle of length  $\ell$ for the admissible cycle $(d,\alpha,\beta)_{\pmb{\pm}}$ with $\beta = a_0 \beta_0$ and we have $\Omega(a_0\omega)=a_0\,\Omega_0(\omega)$.
\end{theorem}
\begin{proof}
	Let $T_0$ and $T$  be the mappings  associated to $(d,\alpha,\beta_0)_{\pmb{\pm}}$ and $(d,\alpha,\beta)_{\pmb{\pm}}$, respectively, with $\beta=a_0\beta_0$, we have  \begin{equation*}\label{map_T_0}
		T_0(n)=\left\{\begin{array}{lcc}
			n/d&\mathrm{if}& \modd{n}{0}{d}\\
			\bigl(\alpha  n+\beta_0 [\kappa_0n]_d\bigr)/d&\mathrm{if}& \notmodd{n}{0}{d},\\
		\end{array}\right.\, \textrm{and}\,
		T(n)=\left\{\begin{array}{lcc}
			n/d&\mathrm{if}& \modd{n}{0}{d}\\
			\bigl(\alpha  n+\beta [\kappa_0n]_d\bigr)/d&\mathrm{if}& \notmodd{n}{0}{d}.\\
		\end{array}\right.
	\end{equation*}
	Let $n$ be any positive integer, as $\modd{a_0}{1}{d}$, then
	$[a_0 n]_d=[n]_d$, 
	$[\kappa_0a_0 n]_d=[\kappa_0n]_d$ and 
	\begin{equation*}\label{map_T_0}
		T(a_0n )=\left\{\begin{array}{lcc}
			(a_0 n)/d&\mathrm{if}& \modd{n}{0}{d}\\
			\bigl(\alpha  a_0 n+a_0\beta_0 [\kappa_0n]_d\bigr)/d&\mathrm{if}& \notmodd{n}{0}{d},\\
		\end{array}\right. 
	\end{equation*}
	It follows that $T(a_0 n)= a_0T_0(n)$, for any positive integer $n$. By induction, assume that at a step  $k$,  we have $T^{(k)}(a_0 n)= a_0T_0^{(k)}(n)$. Then, 
	$$T^{(k+1)}(a_0 n)=T(T^{(k)}(a_0n))=T(a_0T_0^{(k)}(n))=a_0T_0(T_0^{(k)}(n))=a_0 T_0^{(k+1)}(n).
	$$
	Now, assume that $(d,\alpha,\beta_0)_{\pmb{\pm}}$ has a cycle $\Omega_0(\omega)$ of length $\ell$. 
	$$\Omega_0(\omega)= (\omega\rightarrow T_0(\omega)\rightarrow T_0^{(2)}(\omega) \rightarrow\ldots\rightarrow T_0^{(\ell-2)}(\omega)\rightarrow T_0^{(\ell-1)}(\omega) \rightarrow\omega).$$
	As 
	$$a_0 \Omega_0(\omega)= (a_0\omega\rightarrow a_0T_0(\omega)\rightarrow a_0T_0^{(2)}(\omega) \rightarrow\ldots\rightarrow  a_0T_0^{(\ell-1)}(\omega) \rightarrow a_0\omega),$$
	and
	$T^{(\ell)}(a_0\omega) =a_0 T_0^{(\ell)}(\omega) =a_0 \omega$, it follows that 
	$$a_0 \Omega_0(\omega)= (a_0\omega\rightarrow T(a_0\omega)\rightarrow  T^{(2)}(a_0\omega) \rightarrow\ldots\rightarrow  T^{(\ell-1)}(a_0\omega) \rightarrow a_0\omega).$$
	Which shows that, $\Omega(a_0\omega)=a_0 \Omega_0(\omega)$ and $\Omega(a_0\omega)$ is also a cycle for $(d,\alpha,\beta)_{\pmb{\pm}}$ of length $\ell$. 
	\hfill\end{proof}
\begin{example}In this example, we apply the previous theorem. 
	Consider the  triplet $(d,\alpha,\beta)_{\pmb{+}}=(5,6,373769)_{\pmb{+}}$ with its  corresponding  
	mapping  $T:\mathbb{N}\longrightarrow \mathbb{N} $  given  by
	\begin{equation*}
		T(n)=\left\{\begin{array}{ccc}
			n/5&\mathrm{if}& \modd{n}{0}{5}\\ 
			\Bigl(6 n+373769 [ n]_5\Bigr)/5&\mathrm{if}& \notmodd{n}{0}{5}.\\
		\end{array}\right.
	\end{equation*}
	We observe that  here
	$\beta=373769=a_0\,\beta_0= 121\times 3089$ with $\beta_0=3089$ and  $a_0= \modd{121}{1}{5}$. The triplet 
	$(d,\alpha,\beta_0)_{\pmb{+}}=(5,6,3089)_{\pmb{+}}$ is an admissible   triplet   with $\beta_0=3089$ and  its  corresponding 
	mapping  $T_0:\mathbb{N}\longrightarrow \mathbb{N} $  given  by
	\begin{equation*}
		T_0(n)=\left\{\begin{array}{ccc}
			n/5&\mathrm{if}& \modd{n}{0}{5}\\ 
			\Bigl(6 n+3089 [ n]_5\Bigr)/5&\mathrm{if}& \notmodd{n}{0}{5}.\\
		\end{array}\right.
	\end{equation*} It
	has at least the following $33$ cycles of length $\ell=33$:\\
	$
	\begin{array}{ll}
		\Omega(11),  \Omega(16),  
		\Omega(17),  \Omega(21),  
		\Omega(22),  \Omega(23), 
		\Omega(26),  \Omega(27), 
		\Omega(28), \Omega( 29),  
		\Omega(31), \Omega(32),  
		\Omega(33), 
		\Omega(34), \Omega(37),  
		\Omega(38), &\\\Omega(39), 
		\Omega( 43),\Omega( 44), 
		\Omega(49), \Omega(56), 
		\Omega(62), \Omega(68), 
		\Omega(74), \Omega( 81),  
		\Omega(87), \Omega(93), 
		\Omega( 99), \Omega(106),
		\Omega(112),\Omega(118), &\\
		\Omega(124), \Omega(3089), \textrm{ see Example~\ref{Exp(5,6,3089)+}.}&\\
	\end{array}$ 
	According to the previous theorem,  the triplet  $(5,6,373769)_{\pmb{+}}$ is also a admissible triplet and it has at least the following 
	$33$ cycles of length  $\ell=33$: \\
	$
	\begin{array}{ll}
		\Omega(1331), 	\Omega(1936), 	\Omega(2057), 	\Omega(2541), 	\Omega(2662), 	\Omega(2783), 	\Omega(3146),  	\Omega(3267), 	\Omega(3388), 	\Omega(3509),	\Omega(3751), 	&\\
		\Omega(3872), 	\Omega(3993), 	\Omega(4114), 	\Omega(4477), 	\Omega(4598), 	\Omega(4719), 	\Omega(5203), 	\Omega(5324), 	\Omega(5929), 	\Omega(6776), 	\Omega(7502), 	&\\
		\Omega(8228), 	\Omega(8954), 	\Omega(9801), 	\Omega(10527), 	\Omega(11253), 	\Omega(11979), 	\Omega(12826), 	\Omega(13552), 	\Omega(14278), 	\Omega(15004), &\\	\Omega(373769).&
	\end{array}$ This example is the one given in  Item~\ref{Item-2}-Examples~\ref{Example201}.
\end{example}

\subsection{\textbf{The  special admissible triplets $\bf (d,d+1,-1)_{\pmb{+}}$ and $\bf (d,d+1,1)_{\pmb{-}}$.}
}\label{subsubsection:case2.1}
The triplets $(d,d+1,-1)_{\pmb{+}}$ and $ (d,d+1,1)_{\pmb{-}}$  satisfy the condition \eqref{CdtCollatztriplet} and the corresponding mapping  is given by
\begin{equation}\label{mapExample11a}
	T(n)=\left\{\begin{array}{lcc}
		n/d&\mathrm{if}& \modd{n}{0}{d}\\
		\dfrac{(d+1)n-\kappa_0[\kappa_0 n]_d}{d}&\mathrm{if}& \notmodd{n}{0}{d}.\\
	\end{array}\right.
\end{equation}
We have the following result.
\begin{theorem}\label{trivialcycleTheoremB}
	Let $d\geq 2$, we have:
	\begin{enumerate}
		\item  The triplet $(d,d+1,-1)_{\pmb{+}}$  has at least the  $d-1$ trivial cycles  of length $1$:
		\begin{equation}\label{TrivialCycles-(d,d+1,-1,1)}
			\Omega(r)=	(  r\rightarrow  r),\;\forall  r\in\{1,\ldots,d-1\}.
		\end{equation}
	It follows that $\# \mathrm{Cycl}(d,d+1,-1)_{\pmb{+}}\geq d-1$. 
		\item  The triplet $(d,d+1,1)_{\pmb{-}}$  has   at least the trivial cycle  of length $d$:
		\begin{equation}\label{TrivialCycles-(d,d+1,1,-1)}
			\Omega(1)=	( 1\rightarrow 2\rightarrow 3\rightarrow\ldots\rightarrow d \rightarrow  1).
		\end{equation}
		It follows that $\# \mathrm{Cycl}(d,d+1,1)_{\pmb{-}}\geq 1$. 
	\end{enumerate}
\end{theorem}
\begin{proof}
	\begin{enumerate}
		\item For $\kappa_0=1$  and for
		$r\in\{1,\ldots,d-1\}$, 
		$\displaystyle T(r)=\dfrac{(d+1)r-[r]_d}{d}=\dfrac{(d+1)r-r}{d}=\dfrac{dr}{d}=r$, which 
		gives rise to the cycles  $\Omega(r)$ given by \eqref{TrivialCycles-(d,d+1,-1,1)} for $r\in\{1,\ldots,d-1\}$. 
		\item For $\kappa_0=-1$ and for
		$r\in\{1,\ldots,d-1\}$, 
		$\displaystyle T(r)=\dfrac{(d+1)r+[-r]_d}{d}=\dfrac{(d+1)r+(d-r)}{d}$. Then, $T(r)=\dfrac{d(r+1)}{d}=r+1$.
		It follows that $T^{(d)}(1)=1$ which gives rise to the cycle  $\Omega(1)$ given by \eqref{TrivialCycles-(d,d+1,1,-1)}. 
	\end{enumerate}
	\hfill\end{proof}

Let us remark that, the classical Collatz case is obtained for $d=2$ with $\kappa_0=-1$.  The trivial cycle 
$\Omega(1)$ is in the form \eqref{TrivialCycles-(d,d+1,1,-1)}: $\Omega(1) =(1\rightarrow 2 \rightarrow 1)$.  

it appears that  for all  $d\geq 2$,   $(d,d+1,-1)_{\pmb{+}}$ and $ (d,d+1,1)_{\pmb{-}}$ are  strongly  admissible  triplets  and their  respective  orders satisfy  $d-1\leq  \#\mathrm{Cycl}(d,d+1,-1)_{\pmb{+}}$ for $\kappa_0=1$ and $1\leq \#\mathrm{Cycl}(d,d+1,1)_{\pmb{-}}$ for  $\kappa_0=-1$. Their trivial cycles are:  
\begin{itemize}
	\item For $\kappa_0=1$,     $\Omega(r)=(r\rightarrow r)$ for all $r\in\{1,\ldots,d-1\}$ .
	\item For $\kappa_0=-1$, $\Omega(1)=	( 1\rightarrow 2\rightarrow 3\rightarrow\ldots\rightarrow d\rightarrow  1)$.
\end{itemize}

\subsection{
	\textbf{ The admissible triplets $\bf (d,2d-1,1)_{\pmb{+}}$  evoking the Mersene  numbers for $\bf d=2^{p-1}$, $\bf p\geq 2$.} }\label{subsection:case3}

The  triplet 
$ (d,2d-1,1)_{\pmb{+}}$ .satifiens the condition \eqref{CdtCollatztriplet}.
Its corresponding mapping  is given by
\begin{equation}\label{map_case3}
	T(n)=\left\{\begin{array}{lcc}
		n/d&\mathrm{if}& \modd{n}{0}{d}\\
		\dfrac{(2d-1)n+[n]_d}{d}&\mathrm{if}& \notmodd{n}{0}{d}.\\
	\end{array}\right.
\end{equation}
An interesting case may be obtained by setting $d= 2^{p-1}$ where $p\geq 2$. This leads to the  triplet 
$(2^{p-1},2^{p}-1,1)_{\pmb{+}}$ with 
$\alpha =2^p-1$ as a Mersenne number, for   $p\geq 2$.  We recover the classical Collatz  problem by choosing $p=2$.
We have the following result.
\begin{theorem}\label{trivialcycleTheoremC} 
	Let $d_p= 2^{p-1}$ for $p\geq 2$ and let the integers $\alpha_p=2d_p-1=2^{p}-1$, $\beta=1$ and $\kappa_0=1$, then  $(2^{p-1},2^{p}-1,1)_{\pmb{+}}$ is an admissible triplet, it has at least the following trivial cycle of length $p$:
	\begin{equation}\label{TrivialCycles2-case3}
		\Omega(1)=	( 1\rightarrow 2\rightarrow 2^2\rightarrow\ldots\rightarrow 2^{p-1}\rightarrow  1),
	\end{equation}
	It follows that, 
	$\# \mathrm{Cycl}
	(2^{p-1},2^{p}-1,1)_{\pmb{+}}
	\geq 1$. 
\end{theorem}
\begin{proof}
	Let $p\geq 2$, for any integer $k$ with  $0\leq k\leq p-2$, we have
	$$T(2^k)=\dfrac{(2^p-1) 2^k +[2^k]_{d_p}}{d_{p}}=\dfrac{(2^{p}-1)2^k +2^k}{2^{p-1}}=2^{k+1}.$$
	It follows that $T^{(p)}(1)=1$ which gives rise to the trivial cycle \eqref{TrivialCycles2-case3}.
\end{proof}

Let us now  formulate  the 
following  conjecture.
\begin{conjecture}\label{ConjectureC} For all $p\geq 2$, $(2^{p-1},2^{p}-1,1)_{\pmb{+}}$ is an admissible triplet  of order one.
	It is  strongly admissible   only for  $p=2$ and it is a  weakly admissible   for all $p\geq 3$.  Its  trivial cycle   of  length $p$ is:
	$\Omega(1)=	( 1\rightarrow 2\rightarrow 2^2\rightarrow\ldots\rightarrow 2^{p-1}\rightarrow  1)$.
\end{conjecture}
If the last conjecture is true, then  for  $p\geq 3$:
$\#\mathrm{Cycl}(2^{p-1},2^{p}-1,1)_{\pmb{+}}=1$ and $\mathbb{N}=G(1)\cup  G(\infty) $.  The triplet $(2^{p-1},2^{p}-1,1)_{\pmb{+}}$ may have divergent trajectories.

\subsection{\textbf{The triplet
		$\bf (2^p+2^q,2^p+2^{q+1},2^p)_{\pmb{+}}$ with $\bf p\geq q\geq 0$.}}
	For more detail on this case see \cite{Bouhamidi_conj10.6.4}. Consider the mapping $T:=T_{p,q}$, for $p\geq q\geq 0$, given  by
	\begin{equation}\label{mapMainConjecture_dpdq}
		T_{p,q}(n)=\left\{\begin{array}{lcc}
			n/(2^p+2^q)&\mathrm{if}& \modd{n}{0}{(2^p+2^q)}\\
			\dfrac{(2^p+2^{q+1})n+2^p[ n]_{p,q}}{2^p+2^q}&\mathrm{if}& \notmodd{n}{0}{(2^p+2^q)},\\
		\end{array}\right.
	\end{equation}
	here $[\,\,]_{p,q}=[\,\,]_{d_{p,q}}$ stands for the remainder in the Euclidean division by $d_{p,q}=2^p+2^q$.   We observe
	that the case  $p =q=0$  gives the triplet $(2,3,1)_+$  of   the classical Collatz case  \eqref{CollatzMap}   and the case $p=3$ and $q=1$ gives the triplet $(10,12,8)_+$ which is exactly the triplet of our conjecture \ref{Myconjecture}.   For  $p\geq q\geq 1$,  we may simplify the mapping $T_{p,q}$   to
	\begin{equation}\label{mapMainConjecture_dpdq__Simplified}
		T_{p,q}(n)=\left\{\begin{array}{lcc}
			n/(2^p+2^q)&\mathrm{if}& \modd{n}{0}{(2^p+2^q)}\\
			\dfrac{(2^{p-q}+2)n+2^{p-q}[ n]_{p,q}}{2^{p-q}+1}&\mathrm{if}& \notmodd{n}{0}{(2^p+2^q)}.\\
		\end{array}\right.
	\end{equation}
We have the main conjecture.
	\begin{conjecture}\label{MainConjecture2p2q}
		Let  $\mathbb{E}=\{(1,0),(2,1),(2,2),(3,0),(4,0),(5,2),(6,2),(7,0)\}\subset \mathbb{N}_0\times  \mathbb{N}_0$.
		
		For all   integer $p\geq 0$ and for all integer $q$ with $0 \leq q\leq p$: The triplet
		$(2^p+2^q, 2^p+2^{q+1}, 2^p)_+$	  
		is  strongly  admissible. Its  trivial cycle  is $\Omega(2^{p-q})$ of length $2^{p-q}+q+1$ given by:
		\begin{equation}\label{TrivialCycle2p2q}
			\Omega(2^{p-q})=\Bigl(2^{p-q}\rightarrow  2^{p-q+1}\rightarrow
			\cdots  2^{p-1}\rightarrow 
			2^{p}\rightarrow
			2\cdot2^{p}\rightarrow
			3\cdot2^{p}\rightarrow
			\ldots\rightarrow (2^{p-q}+1)\cdot2^{p}\rightarrow 2^{p-q}\Bigr).
		\end{equation}
		Furthermore.
		\begin{enumerate}
			\item 	If $(p,q)\not\in \mathbb{E}$, then  the triplet 
			$(2^p+2^q, 2^p+2^{q+1}, 2^p)_+$  is of order one, its unique trivial cycle is  the one given by 
			\eqref{TrivialCycle2p2q}. So, for all integer $n\geq 1$,  there exists an integer $k\geq 0$ such that $T_{p,q}^{(k)}(n)=2^{p-q}$.
			\item  If $(p,q)\in \mathbb{E}$, then:
			\begin{enumerate}
				\item  If $(p,q)\not=(5,2)$, then  the triplet  
				$(2^p+2^q, 2^p+2^{q+1}, 2^p)_+$   is  of order two.   Its first trivial cycle is   the one given by 
				\eqref{TrivialCycle2p2q} and its second trivial cycle is  $\Omega(\omega_{p,q})$ the one starting  by $\omega_{p,q}$ given bellow. So, for all $n\geq 1$ there exists an integer $k\geq 0$ such that $T^{(k)}_{p,q}(n)\in \{2^{p-q},\omega_{p,q}\}$. 
				\item  If $(p,q)=(5,2)$,  the triplet $(36,40,30)_+$ is of order three.  Its first  trivial cycle is  given by 	\eqref{TrivialCycle2p2q}  and is $\Omega(8)=\bigl(8\rightarrow16\rightarrow32\rightarrow64\rightarrow96\rightarrow128\rightarrow160\rightarrow192\rightarrow224\rightarrow256\rightarrow288\rightarrow8\bigr)$. Its second and third cycles are
				$\Omega(\omega^{(1)}_{5,2})$ and $\Omega(\omega^{(2)}_{5,2})$ starting  by $\omega^{(1)}_{5,2}=76200$  and  $\omega^{(1)}_{5,2}=87176$  of length $70$ and  $35$, receptively. So, for all integer $n\geq 1$,  there exists an integer $k\geq 0$ such that 
				$T_{5,2}^{(k)}(n)\in\{8,76200,87176\}$.
			\end{enumerate}
		\end{enumerate}			
		The cycles  $\Omega(\omega_{1,0})$, $\Omega(\omega_{2,1})$, $\Omega(\omega_{2,2})$, $\Omega(\omega_{3,0})$, $\Omega(\omega_{4,0})$, $\Omega(\omega^{(1)}_{5,2})$, $\Omega(\omega^{(2)}_{5,2})$, $\Omega(\omega_{6,2})$ and $\Omega(\omega_{7,0})$
		starting by $\omega_{1,0}=14$, $\omega_{2,1}=74$,  $\omega_{2,2}=67$,  $\omega_{3,0}=280$, $\omega_{4,0}=1264$,
		$\omega^{(1)}_{5,2}=76200$, $\omega^{(2)}_{5,2}=87176$, $\omega_{6,2}=1264$ and $\omega_{7,0}=3027584$, respectively and of length
		$9$, $7$, $6$, $21$,  $49$, $70$, $35$, $69$, $630$  respectively are given as following:
		$$\begin{array}{ll}
			\Omega(\omega_{1,0})=&
			\bigl(14\rightarrow 20\rightarrow  28\rightarrow 38\rightarrow 52\rightarrow 70\rightarrow  94 \rightarrow 126\rightarrow 42\rightarrow 14\bigr),\\
			\Omega(\omega_{2,1})=&	\bigl(74\rightarrow 100\rightarrow  136\rightarrow 184\rightarrow 248\rightarrow 332\rightarrow  444 \rightarrow 74 \bigr),\\
			\Omega(\omega_{2,2})=&
			\bigl(67\rightarrow102\rightarrow156\rightarrow236\rightarrow356\rightarrow536\rightarrow67\bigr),\\
			\Omega(\omega_{3,0})=&
			\bigl(280\rightarrow 312\rightarrow 352\rightarrow 392\rightarrow 440\rightarrow 496\rightarrow 552\rightarrow 616\rightarrow 688\rightarrow
			768\rightarrow\\
			& 856\rightarrow952\rightarrow1064\rightarrow 1184\rightarrow 1320\rightarrow 1472\rightarrow 1640\rightarrow
			1824\rightarrow 2032\rightarrow \\
			&2264 \rightarrow 2520 \rightarrow280),\\
			\Omega(\omega_{4,0})=&\bigl(
			1264\rightarrow 1344\rightarrow 1424\rightarrow 1520\rightarrow 1616\rightarrow 1712\rightarrow 1824\rightarrow 
			1936\rightarrow 2064\rightarrow\\
			&2192\rightarrow 2336\rightarrow  2480\rightarrow 2640\rightarrow 2800\rightarrow
			2976 \rightarrow 3152\rightarrow3344\rightarrow 3552\rightarrow 3776\\
			&\rightarrow 4000\rightarrow4240\rightarrow 
			4496\rightarrow
			4768\rightarrow 5056\rightarrow 5360\rightarrow 5680\rightarrow 6016\rightarrow 6384\\
			&\rightarrow6768\rightarrow 7168 \rightarrow 7600\rightarrow 8048\rightarrow 8528\rightarrow
			9040\rightarrow 9584\rightarrow 
			10160\rightarrow 10768	\\
			& \rightarrow 11408\rightarrow 12080\rightarrow 12800\rightarrow 13568\rightarrow 
			14368\rightarrow 15216\rightarrow 16112\rightarrow 17072\rightarrow	\\
			&  18080\rightarrow 19152\rightarrow
			20288\rightarrow 21488\rightarrow 1264\bigr),\\
			\Omega(\omega^{(1)}_{5,2})=&
			\bigl(76200\rightarrow84688\rightarrow94112\rightarrow104576\rightarrow116224\rightarrow129152\rightarrow143520\rightarrow159488\\
			&\rightarrow177216\rightarrow196928\rightarrow218816\rightarrow243136\rightarrow270176\rightarrow300224\rightarrow333600\rightarrow\\&370688\rightarrow411904\rightarrow457696\rightarrow508576\rightarrow565088\rightarrow627904\rightarrow697696\rightarrow775232\\&\rightarrow861376\rightarrow957088\rightarrow1063456\rightarrow1181632\rightarrow1312928\rightarrow1458816\rightarrow1620928\\&\rightarrow1801056\rightarrow2001184\rightarrow2223552\rightarrow2470624\rightarrow2745152\rightarrow3050176\rightarrow3389088\\&\rightarrow3765664\rightarrow4184096\rightarrow4649024\rightarrow5165600\rightarrow5739584\rightarrow6377344\rightarrow7085952\\&\rightarrow196832\rightarrow218720\rightarrow243040\rightarrow270048\rightarrow300064\rightarrow333408\rightarrow370464\rightarrow\\&411648\rightarrow457408\rightarrow508256\rightarrow564736\rightarrow627488\rightarrow697216\rightarrow774688\rightarrow860768\\&\rightarrow956416\rightarrow1062688\rightarrow1180768\rightarrow1311968\rightarrow1457760\rightarrow1619744\rightarrow1799744\\&\rightarrow1999744\rightarrow2221952\rightarrow2468864\rightarrow2743200\rightarrow76200\bigr),\\
			\Omega(\omega^{(2)}_{5,2})=&
			\bigl(87176\rightarrow96880\rightarrow107648\rightarrow119616\rightarrow132928\rightarrow147712\rightarrow164128\rightarrow182368\\& \rightarrow202656\rightarrow225184\rightarrow250208\rightarrow278016\rightarrow308928\rightarrow343264\rightarrow381408\\&\rightarrow423808\rightarrow470912\rightarrow523264\rightarrow581408\rightarrow646016\rightarrow717824\rightarrow797600\\&\rightarrow886240\rightarrow984736\rightarrow1094176\rightarrow1215776\rightarrow1350880\rightarrow1500992\rightarrow1667776\\&\rightarrow1853088\rightarrow2059008\rightarrow2287808\rightarrow2542016\rightarrow2824480\rightarrow3138336\rightarrow87176\bigr),\\
			\Omega(\omega_{6,2})=&\bigl(1264\rightarrow1376\rightarrow1472\rightarrow1600\rightarrow1728\rightarrow1856\rightarrow1984\rightarrow2112\rightarrow2240\rightarrow2432\\
			&\rightarrow2624\rightarrow2816\rightarrow3008\rightarrow3200\rightarrow3392\rightarrow3648\rightarrow3904\rightarrow4160\rightarrow4416\rightarrow4736\\&\rightarrow5056\rightarrow5376\rightarrow5696\rightarrow6080\rightarrow6464\rightarrow6848\rightarrow7296\rightarrow7744\rightarrow8256\rightarrow8768\\&\rightarrow9344\rightarrow9920\rightarrow10560\rightarrow11200\rightarrow11904\rightarrow12608\rightarrow13376\rightarrow14208\rightarrow15104\\&\rightarrow16000\rightarrow16960\rightarrow17984\rightarrow19072\rightarrow20224\rightarrow21440\rightarrow22720\rightarrow24064\rightarrow\\&25536\rightarrow27072\rightarrow28672\rightarrow30400\rightarrow32192\rightarrow34112\rightarrow36160\rightarrow38336\rightarrow40640\\&\rightarrow43072\rightarrow45632\rightarrow48320\rightarrow51200\rightarrow54272\rightarrow57472\rightarrow60864\rightarrow64448\rightarrow\\&68288\rightarrow72320\rightarrow76608\rightarrow81152\rightarrow85952\rightarrow1264\bigr),\\
			\Omega(\omega_{7,0})=&\bigl(
			3027584\rightarrow3051136\rightarrow3074816\rightarrow3098752\rightarrow 3122816\rightarrow3147136\rightarrow3171584\rightarrow\\&3196288\rightarrow3221120\rightarrow3246208\rightarrow3271424\rightarrow3296896\rightarrow3322496\rightarrow3348352\rightarrow\ldots\ldots\\&
			\ldots\rightarrow
			367160576\rightarrow370006784\rightarrow372875136\rightarrow375765760\rightarrow378678784\rightarrow381614336\rightarrow\\& 384572672\rightarrow387553920\rightarrow390558336\rightarrow3027584\bigr).
		\end{array}$$		
	\end{conjecture}
	
	For the special case, where $p=q$ we obtain the triplet 
		$(d_p,\alpha_p,\beta_p)_+:=(2^{p+1},3\times 2^p,2^p)_+$.
The corresponding mapping $T:=T_{p}$, for $p\geq 0$, is given  by
	\begin{equation}\label{mapMainConjecture_dpSimple}
		T_{p}(n)=\left\{\begin{array}{lcc}
			n/2^{p+1}&\mathrm{if}& \modd{n}{0}{2^{p+1}}\\
			\dfrac{3n+[ n]_{p}}{2}&\mathrm{if}& \notmodd{n}{0}{2^{p+1}},\\
		\end{array}\right.
	\end{equation}
	here $[\,\,]_{p}=[\,\,]_{d_{p}}$ stands for the remainder in the Euclidean division by $d_{p}=2^{p+1}$.  It is clear that the classical Collatz mapping is obtained for $p=0$.   From Conjecture~\ref{MainConjecture2p2q},  we deduce the following special case conjecture.

	\begin{conjecture}\label{MainConjecture2p}
		\begin{itemize}
			\item   For all $p\geq 0$ with $p\not=2$  the triplet  $(2^{p+1},3\times 2^{p},2^{p})_+$, associated to the mapping $T_p$ given by \eqref{mapMainConjecture_dpSimple}, is a strongly  admissible  triplet of order one, its unique  trivial cycle of length $p+2$ is:
			$$\Omega(1)=\bigl(1\rightarrow  2\rightarrow 2^2
			\rightarrow\cdots \rightarrow 
			2^{p+1}\rightarrow 1\bigr).$$
			For all integer $n\geq 1$, there exists an integer $k\geq 0$ such that $T^{(k)}_p(n)=1$.\\
			\item For $p=2$, the corresponding strongly admissible triplet $(8,12,4)_+$ is of order  $2$, its first trivial cycle is $\Omega(1)=\bigl(1\rightarrow2\rightarrow4\rightarrow8\rightarrow1\bigr)$  and its second trivial cycle is 
			$\Omega(67)=
			\bigl(67\rightarrow102\rightarrow156\rightarrow236\rightarrow356\rightarrow536\rightarrow67\bigr)$. For all $n\geq 1$, there exists an integer $k\geq 0$ such that $T_2^{(k)}(n)\in\{1,67\}$.
		\end{itemize}
	\end{conjecture}

\section{Lower bound lengths  for cycles  associated to an admissible  triplet}
In this section,  we consider an admissible 
triplet $(d,\alpha,\beta)_{\pmb{\pm}}$ with $\alpha>d\geq 2$. We restrict our attention to the case where the numbers  $d$ and $\alpha$ are supposed to be coprime, $cgd(d,\alpha)=1$ and $\beta>0$.  
Let $\xi$ be an irrational real number.  The number  $\mu= \mu(\xi)$  is said to be an effective irrationality measure  of $\xi$ (see \cite{Rhin1987} for more details)  if for all $\varepsilon>0$,  there exists an integer   $q_0(\varepsilon)>0$  (effectively computable) such that   
\begin{equation}\label{muIrrationalityMeasureepsilonF1}
	\forall (p,q) \in \mathbb{Z}\times \mathbb{N},\quad   q>q_0(\varepsilon) \Longrightarrow
	\Bigl|\dfrac{p}{q}-\xi\Bigr|>\dfrac{1}{q^{\mu+\varepsilon}}.
\end{equation}
By definition, see 	\cite{Bugeaud2015,Waldschmidt2000},
two positive rational numbers are said to be multiplicatively independent if the quotient of their logarithms is irrational. In the following, we fix  $\xi$  to be 
the  real number $\xi=\log_d(\alpha):=\dfrac{\log(\alpha)}{\log(d)}$.  So, the number $\xi$ is  irrational.
Indeed, suppose that  $\xi$ is a rational number, then there exist a pair of integers $(p,q)\in\mathbb{N}^2$ such that $\alpha^q=d^p$, which is impossible since $\alpha$ and $d$ are assumed to be coprime numbers.  We have the following result.
\begin{theorem}
	The number $\xi=\dfrac{\log(\alpha)}{\log(d)}$ is a transcendental real number and 
	the effective irrationality measure  $\mu(\xi)$ of $\xi$ is a finite number. More precisely, there exists a constant $C$ such that: 
	\begin{equation}\label{irrationalitymeasureInquality}
		2\leq \mu(\xi)\leq C(\log \alpha )(\log d ).
	\end{equation}
	
\end{theorem}
\begin{proof}
	As $\xi$ is irrational real number and $d^\xi=\alpha$ is an algebraic number, then according to the Gelfond-Schneider theorem, we get that $\xi$ is transcendental real number.  Let  $a_1=\alpha>a_2=1$ and $b_1=d>b_2=1$,  as  $\xi$ 
	is irrational, then $a_1/a_2$ and $b_1/b_2$ are multiplicatively independent.  According to  Theorem~1.1 given in 
	\cite{Bugeaud2015}, see also  \cite{Waldschmidt2000}, there exists an absolute, effectively computable, constant $C$
	such that
	$$\mu\Bigl(
	\dfrac{\log(a_1/a_2)}{log(b_1/b_2) }\Bigr)\leq C(\log a_1 )(\log b_1 ).$$
	It follows that  $\mu(\xi)\leq C(\log \alpha )(\log d )$. The  convergents of the continued fraction of the irrational number $\xi$ imply that $\mu(\xi) \geq  2$. Thus, the effective irrationality measure  of $\xi$  satisfies the inequality \eqref{irrationalitymeasureInquality}. 
	\hfill\end{proof}

According to \eqref{muIrrationalityMeasureepsilonF1},
for $\varepsilon=1$,  there exists an integer $\widehat{q}_0(\xi)>0$ such that
\begin{equation}\label{muIrrationalityMeasureepsilonF2}
	\forall (p,q) \in \mathbb{Z}\times \mathbb{N},\quad   q>\widehat{q}_0(\xi) \Longrightarrow
	\Bigl|\dfrac{p}{q}-\xi\Bigr|>\dfrac{1}{q^{\mu(\xi)+1}}\Longrightarrow \Bigl|p-q\xi\Bigr|>\dfrac{1}{q^{\mu(\xi)}}.
\end{equation}
Let us now recall some properties for the continued fractions which may be found in the literature of the theory of rational approximation, see for instance  \cite{Waldschmidt2015}.
As $\xi$ is an irrational number, there is a unique representation of $\xi$ as an infinite simple continued fraction:
$\xi=[a_0,a_1,a_2,a_3,\ldots]$.
The integer numbers $a_n$ are called  the partial quotients  and  the rational numbers 
$\dfrac{p_n}{q_n} = [a_0,a_1,\ldots,a_n]$
are called  the convergents of $\xi$, where $gcd(p_n,q_n)=1$. The numbers
$x_n= [a_n,a_{n+1},\ldots]$
are called the complete quotients.  For all $n \geq 0$, we have
$$\xi= [a_0,a_1,\ldots,a_n,x_{n+1}]=\dfrac{x_{n+1}p_n+p_{n-1}}{x_{n+1}q_n+q_{n-1}}.$$
The integers $p_n,q_n$ are obtained recursively as follows:
$$\begin{array}{rcl}
	p_{n}&=&a_np_{n-1}+p_{n-2},\\
	q_{n}&=&a_nq_{n-1}+q_{n-2},\\
\end{array}
$$
with the initial values  $p_{-1}=1$, $p_{-2}=0$ and   $q_{-1}=0$, $q_{-2}=1$.
The sequences $(p_n)$ and $(q_n)$ are strictly increasing unbounded sequences.  It is well known that 
the fractions $p_n/q_n$ satisfy the following properties:
For any pair of integers $(p,q)\in\mathbb{N}^2$ and for $n\geq 0$ with $1<q<q_n$ we have:
\begin{equation}\label{Lemma1ContinuedFractions} 
	\dfrac{1}{q_n+q_{n+1}}<	\Bigl|p_n-q_n\xi\Bigr|<\Bigl|p-q\xi\Bigr|.
\end{equation}
It is also well known that  $(p_{2n}/q_{2n})_n$ is an increasing 
sequence and that $(p_{2n+1}/q_{2n+1})_n$ is a decreasing 
sequence  such that
\begin{equation}\label{rapportSequenceOfConvergents}
	\dfrac{p_0}{q_0}<\dfrac{p_2}{q_2}<\dfrac{p_4}{q_4}<\ldots<
	\dfrac{p_{2n}}{q_{2n}}<\dfrac{p_{2n+2}}{q_{2n+2}}<\ldots<\xi<\ldots<	\dfrac{p_{2n+1}}{q_{2n+1}}<\dfrac{p_{2n-1}}{q_{2n-1}}<\ldots<
	\dfrac{p_5}{q_5}<	\dfrac{p_3}{q_3}<\dfrac{p_1}{q_1},
\end{equation}
and $\displaystyle \lim_{n \longrightarrow \infty}\dfrac{p_{n}}{q_{n}}=\xi$.

Let $\Omega$ be an eventual  cycle for  the admissible 
triplet $(d,\alpha,\beta)_{\pmb{\pm}}$  and let $T$ denote  its associated  operator. We denote by 
$\Omega_d:=\{n\in\Omega\;:\; \modd{n}{0}{d}\}$
the subset of $\Omega$  consisting of all the elements  in $\Omega$ that  are congruent to zero modulo $d$ and we denote by $\overline{\Omega}_d:=\{n\in\Omega\;:\; \notmodd{n}{0}{d}\}$, the complement of $\Omega_d$ in $\Omega$. 
Let $L=\#\Omega$ and $\overline{K}=\#\overline{\Omega}_d$ denote the cardinality  of $\Omega$ and $\overline{\Omega}_d$, respectively. Let
$\max(\Omega)$
and  $\min(\Omega)$ denote  the  greatest and smallest    elements  in  the cycle  $\Omega$, respectively.  We have the following Lemma which is an extension to the result given  in \cite{Eliahou1993} in the case of the classical Collatz  problem. 
\begin{lemma}\label{LemmaexistenceCycle}
	A  necessary condition that a
	cycle $\Omega$  associated to  the admissible 
	triplet $(d,\alpha,\beta)_{\pmb{\pm}}$
	exists  is that  the following inequalities hold:
	\begin{equation}\label{InequalityCardCycle1}
		0<\overline{K}\log_d\left(
		1+\dfrac{\beta}{\alpha\max(\Omega)}
		\right)
		< L-\overline{K}\;\xi \leq
		\sum_{n\in\overline{\Omega}_d}
		\log_d\left(
		1+\dfrac{\beta\;(d-1)}{\alpha\;n}\right)\leq
		\dfrac{\beta\;(d-1)}{\alpha\;\log(d)}\sum_{n\in\overline{\Omega}_d}\dfrac{1}{n},
	\end{equation}
	and  
	\begin{equation}\label{InequalityCardCycle2}
		0<	L-\overline{K}\;\xi\leq 
		\overline{K}\;		\log_d\left(
		1+\dfrac{\beta\;(d-1)}{\alpha\;\min(\Omega)}\right)
		\leq
		\dfrac{\overline{K}\;\beta\;(d-1)}{\alpha\;\log(d)\;\min(\Omega)}.
	\end{equation}
\end{lemma}

\begin{proof}
	Suppose that a  cycle  $\Omega$ of length $L$ exists, then 
	$\displaystyle \prod_{n\in\Omega}n=\prod_{n\in\Omega}T(n)$.
	It follows that
	$$1=\prod_{n\in\Omega}\left(\dfrac{T(n)}{n}\right)=\left(\prod_{n\in\Omega_d}\dfrac{T(n)}{n}\right)\times \left(\prod_{n\in\overline{\Omega}_d}\dfrac{T(n)}{n}\right)=\dfrac{1}{d^L}\times 	\prod_{n\in \overline{\Omega}_d}\left(\alpha+\frac{\beta[\kappa_0 n]_d}{n}\right).$$
	Then
	\begin{equation*}\label{InequalituPower}
		\left(\alpha+\dfrac{\beta}{\max(\Omega)}\right)^{\overline{K}}\leq\prod_{n\in \overline{\Omega}_d}\left(\alpha+\frac{\beta}{n}\right)
		\leq 
		\prod_{n\in \overline{\Omega}_d}\left(\alpha+\frac{\beta[\kappa_0 n]_d}{n}\right)=d^L
		\leq
		\prod_{n\in \overline{\Omega}_d}\left(\alpha+\frac{\beta(d-1)}{n}\right),
	\end{equation*}
	Passing to  the logarithms, we obtain the following inequality
	$$0<\overline{K}\log_d\left(
	1+\dfrac{\beta}{\alpha\max(\Omega)}
	\right)
	\leq  L-\overline{K}\;\xi \leq
	\sum_{n\in\overline{\Omega}_d}
	\log_d\left(
	1+\dfrac{\beta\;(d-1)}{\alpha\;n}\right).
	$$		
	Using  the fact that $\log(1+x)<x$ for $x>0$,  we get the inequalities \eqref{InequalityCardCycle1}.
	Using  again in \eqref{InequalityCardCycle1} the fact that $\log(1+x)<x$ for $x>0$, we obtain the required inequality	
	(\ref{InequalityCardCycle2}).
	\hfill\end{proof}

In what follows, we establish a lower bound for the length of the eventual cycle derived from Hurwitz's theorem on Diophantine approximation. We have the result.

\begin{theorem}\label{Lower_Hurwitz_Bound}
	Assume that  a cycle $\Omega$ associated to  the admissible 
	triplet $(d,\alpha,\beta)_{\pmb{\pm}}$ exists 
	such that  $\min(\Omega)\geq M_0$, where $M_0$ is a sufficiently large integer. Then, the length $L$ of $\Omega$ satisfies
	\begin{equation}\label{Lower_Bound_Hurwitz}
		L \geq \mu_0 \sqrt{M_0}, \quad \text{where} \quad \mu_0 = \sqrt{\dfrac{\alpha \log(d)}{\beta(d-1)\sqrt{5}}}.
	\end{equation}
\end{theorem}
\begin{proof}
	Let  $\displaystyle c_0=\dfrac{\alpha\log(d)M_0}{\beta(d-1)\overline{K}^2} $,
	from Lemma~\ref{LemmaexistenceCycle}, we have
	$$\Bigl|\dfrac{L}{\overline{K}}-\xi\Bigr|\leq \dfrac{\beta(d-1)}{\alpha\log(d)M_0}=\dfrac{\beta(d-1)}{\alpha\log(d)M_0}=\dfrac{1}{c_0 \overline{K}^2},$$
According to Hurwitz's theorem on Diophantine approximation, see for instance \cite{hardywright2008}, there  are infinitely many rational numbers $p/q$  such that
	$$\Bigl|\dfrac{p}{q}-\xi\Bigr|\leq \dfrac{1}{\sqrt{5}\,q^2},$$
	and $\sqrt{5}$ is the optimal constant  value satisfying the last inequality. Then,  $c_0\leq \sqrt {5}$. It follows that
	$$L^2\geq \overline{K}^2\geq \dfrac{\alpha\log(d)M_0}{\beta(d-1)\sqrt{5}}=\mu_0^2 M_0,$$
	which gives the lower bound \eqref{Lower_Bound_Hurwitz}.
\end{proof}

Also from  Lemma~\ref{LemmaexistenceCycle}, we may derive the following theorems.

\begin{theorem}\label{theorem2_K>=}
	Let $(p_n/q_n)$ be the sequence of the convergents of $\xi=\log_d(\alpha)$. 
	A  necessary condition that a
	cycle  of length $L$ associated to  the admissible 
	triplet $(d,\alpha,\beta)_{\pmb{\pm}}$
	exists  is that there exists an integer  $n_0>0$ such that  for all $n\geq n_0$
	the following inequality holds:
	\begin{equation}\label{Kd>qnmu} L\geq\min\Bigl(q_{n},\dfrac{\gamma_0\;  \min(\Omega)}{q_n^{\mu(\xi)}}\Bigr),
	\end{equation}
	where  $\mu(\xi)$ is the irrationality measure of $\xi$ given by 
	\eqref{muIrrationalityMeasureepsilonF1} and
	$\displaystyle \gamma_0=	\dfrac{\alpha\;\log(d)\;}{\beta\;(d-1)}$. 
\end{theorem}
\begin{proof}
	Let   $p_n/q_n$ be the convergents to $\xi$. Then,  $(q_n)_n$ is strictly increasing unbounded sequence. It follows that there exists an integer $n_0>0$ such that  $\forall n>n_0$, we have $q_n>\widehat{q}_0(\xi)$, where $\widehat{q}_0(\xi)$ is the integer defined  by \eqref{muIrrationalityMeasureepsilonF2}.
	It follows  that 	$\Bigl|p_n-q_n\xi\Bigr|>\dfrac{1}{q_n^{\mu(\xi)}}$.
	If $\overline{K}\geq q_n$, the required inequality  \eqref{Kd>qnmu} is trivial.  Then, suppose that  $1<\overline{K}<q_n$. According to \eqref{Lemma1ContinuedFractions} and to
	Lemma~\ref{LemmaexistenceCycle},  for $n\geq  n_0$, we have
	$L-\overline{K}\xi=\Bigl|L-\overline{K}\xi\Bigr|\geq
	\Bigl|p_n-q_n\xi\Bigr|>\dfrac{1}{q_n^{\mu(\xi)}}$. Then, we have $0\leq L-\overline{K}\;\xi \leq \dfrac{\overline{K}\;\beta\;(d-1)}{\alpha\;\log(d)\;\min(\Omega)}$.  It follows, that
	$\displaystyle \dfrac{\overline{K}\;\beta\;(d-1)}{\alpha\;\log(d)\;\min(\Omega)}\geq \dfrac{1}{q_n^{\mu(\xi)}}$,  which gives the required result.
	\hfill\end{proof}

Now, we will use similar arguments  to those of the previous theorem to deduce the following result.   We point out that  similar results have been given  by Crandall \cite{Crandall1978}  in the particular case of  the classical Collatz triplet $(2,3,1)_{\pmb{
		+}}$ .  From the  proposed result we may easily derive, as it will be seen, an algorithm that allow  the computation of   a lower bound of $\#\Omega$, see  Algorithm~1 and  the following section.

\begin{theorem}\label{Main_theorem_4.3}
	Let $(p_n/q_n)$ be the sequence of the convergents of $\xi=\log_d(\alpha)$.  A  necessary condition that a cycle $\Omega$ associated to  the admissible 
	triplet $(d,\alpha,\beta)_{\pmb{\pm}}$
	such that  $\min(\Omega)\geq M>0$ exists  is that:
	\begin{equation}\label{Corollary-CardOmega>d^p}
		\overline{K}\geq  R_{\infty}:= \max_{n\geq 1}\Bigl(R_n(M)\Bigr),
	\end{equation}
	where 
	$\displaystyle R_n(M)=\Bigl\lfloor\min\Bigl(q_n,
	\dfrac{\gamma_0\;M}{(q_n+q_{n+1})}\Bigr)\Bigr\rfloor+1$,   the constant 
	$\displaystyle \gamma_0=	\dfrac{\alpha\;\log(d)}{\beta\;(d-1)}$ and  $\lfloor \,.\,\rfloor$ is the floor function.
\end{theorem}
\begin{proof} 
	First, we will show that the following inequality
	\begin{equation}\label{Kd>minOmega2}
		\overline{K}\geq \min\Bigl(q_n,\dfrac{\gamma_0\;\min(\Omega)}{(q_n+q_{n+1})}\Bigr),
	\end{equation} 
holds for all $n\geq 1$. Indeed,
	let $n\geq 1$, if $ \overline{K}\geq q_n$, then the  required inequality \eqref{Kd>minOmega2}  is obvious. Assume that $ 1<\overline{K}< q_n$, it follows from the inequality  \eqref{Lemma1ContinuedFractions}  that 
	$$L- \overline{K}\xi=|L- \overline{K}\xi|>|p_n-q_n\xi|>\dfrac{1}{q_n+q_{n+1}}.$$
	From Lemma~\ref{LemmaexistenceCycle}, we deduct 
	$\displaystyle \dfrac{\overline{K}\;\beta\;(d-1)}{\alpha\log(d)\min(\Omega)}>\dfrac{1}{q_n+q_{n+1}}$.
	Which prove the required inequality \eqref{Kd>minOmega2}.
	Now, as the sequence $(q_n)_n$ is increasing and unbounded and the sequence 
	$\Bigl(\dfrac{\gamma_0\;\min(\Omega)}{q_n+q_{n+1}}\Bigr)_{n\geq 1}$
	is decreasing and converges to zero. It follows that the sequence 	$\Bigl(R_n(M)
	\Bigr)_n$ 
	is first increasing until it reaches its maximum and after it is decreasing to $1$. So, there exists a positive integer $n_0\geq 1$ such that
	$\displaystyle 
	\max_{n\geq 1}
	\Bigl(R_n(M)
	\Bigr)= R_{n_0}(M)$.
	The required inequality \eqref{Corollary-CardOmega>d^p}  is now an immediate consequence of the  inequality \eqref{Kd>minOmega2}.
	\hfill\end{proof}

Another lower bound of the cardinality $\#\Omega$ for an eventual cycle $\Omega$, may be obtained by  following the idea given  by Eliahou in \cite{Eliahou1993}  using Farey  fractions.   Before that, let us briefly recall some definitions and properties related to  the Farey  fractions, see  \cite{Richards1981,NivenZuckermanMontgomery1991}.  A pair of fractions,  $(a/b,c/d)$ (with $a$, $b$, $c$ and $d$ non-negative integers) is called a Farey pair if $bc-ad=\pm1$.  If 
$\dfrac{a}{b}<\dfrac{c}{d}$, then $bc-ad=1$,  the median of this Farey pair $(a/b,c/d)$
is defined as $(a+c)/(b+d)$. A straightforward calculation shows that 
$$\dfrac{a}{b}<\dfrac{a+c}{b+d}<\dfrac{c}{d},$$
and that  $(a/b,(a+c)/(b+d))$ and  $((a+c)/(b+d),c/d)$
form also two Farey pairs. We may continue the process called Farey process to generate squeezed Farey pairs:
$$
\begin{array}{c}
	\dfrac{a}{b}<\dfrac{c}{d}\\
	\dfrac{a}{b}<\dfrac{a+c}{b+d}<\dfrac{c}{d}\\
	\dfrac{a}{b}<\dfrac{2a+c}{2b+d}<\dfrac{a+c}{b+d}<
	\dfrac{a+2c}{b+2d}<\dfrac{c}{d},\\
	\dfrac{a}{b}<\dfrac{3a+c}{3b+d}<\dfrac{2a+c}{2b+d}<\dfrac{3a+2c}{3b+2d}<\dfrac{a+c}{b+d}<
	\dfrac{2a+3c}{2b+3d}<
	\dfrac{a+2c}{b+2d}<\dfrac{a+3c}{b+3d}<\dfrac{c}{d},
\end{array}
$$
and so on.  Thus, the  sequence generated  by the Farey process is called a Farey sequence.
We observe that, if   $(a/b,c/d)$  is a Farey pair such that $\dfrac{a}{b}<\dfrac{c}{d}$,  then there are  infinite number of fractions of the form
$\dfrac{\alpha a+\beta c}{\alpha b+\beta d}$ such that
$\displaystyle \dfrac{a}{b}<\dfrac{\alpha a+\beta c}{\alpha b+\beta d}<\dfrac{c}{d}$,
with $\alpha\geq 1$ and $\beta\geq 1$ are positive integers.   Usually it is 
sufficient to construct a farey sequence between the Farey pair $(0/1,1/1)$. Then, to obtain a Farey sequence  between any other Farey pair $(a/b,c/d)$, we use the
bijective mapping $f:\mathbb{Q}\cap [\dfrac{0}{1},\dfrac{1}{1}]\longrightarrow \mathbb{Q}\cap [\dfrac{a}{b},\dfrac{c}{d}]$
given by $f(t)=\dfrac{(c-a)t+a}{(d-b)t+b}$ which transform  one-to-one a Farey sequence between  $(0/1,1/1)$ to one between  $(a/b,c/d)$. The inverse mapping of $f$ is obviously   $f^{-1}:\mathbb{Q}\cap[\dfrac{a}{b},\dfrac{c}{d}] \longrightarrow \mathbb{Q}\cap [\dfrac{0}{1},\dfrac{1}{1}]$ given by
$f^{-1}(r)=\dfrac{br-a}{(c-dr)+(br-a)}$.
The following Lemma gives a reciprocal property.  The proof of this Lemma for the Farey pair $(0/1,1/1)$ is given
in \cite{Richards1981}  and in \cite{Eliahou1993} for a general Farey pair $(a/b,c/d)$.   However, for fun of it and for a nice readability of this paper, we well  rewrite this proof given in \cite{Eliahou1993} with slight modifications.

\begin{lemma}\label{Eliahou-Richards-Lemma}
	Let  $\dfrac{a}{b}<\dfrac{c}{d}$ be a Faray pair. Then, every rational number   $\dfrac{x}{y}$ in lowest terms, with  $\dfrac{a}{b}<\dfrac{x}{y}<\dfrac{c}{d}$, appears at some stage of the Farey process, namely, there exist 
	two positive integers $\alpha\geq 1$ and  $\beta\geq 1$
	such that  $\dfrac{x}{y}=\dfrac{\alpha a+\beta c}{\alpha b+\beta d}$.  In particular, $x\geq a+c$ and $y\geq b+d$.
\end{lemma}
\begin{proof}  Let us consider the bijective mapping $f$ and its inverse $f^{-1}$ given above.  Suppose that $\dfrac{a}{b}<\dfrac{x}{y}<\dfrac{c}{d}$, then
	we have
	$$f^{-1}\Bigl(\dfrac{x}{y}\Bigr)=\dfrac{b\dfrac{x}{y}-a}{(c-d\dfrac{x}{y})+(b\dfrac{x}{y}-a)}=\dfrac{bx-ay}{(cy-dx)+(bx-ay)}=\dfrac{\beta}{\alpha+\beta},$$
	where $\alpha= cy-dx>0$ and $\beta=bx-ay>0$ are positive integers. Thus
	$\dfrac{x}{y} =f\Bigl(\dfrac{\beta}{\alpha+\beta}\Bigr)=\dfrac{\alpha a+\beta c}{\alpha b+\beta d}$.
	\hfill\end{proof}

It is well known that two consecutive convergents $(p_n/q_n)$ and $(p_{n+1}/q_{n+1}) $ satisfy the relation $p_nq_{n+1}-p_{n+1}q_{n}=(-1)^{n+1}$, namely they form a Farey pair.

\begin{theorem}\label{Main_theorem_4.4}
Let $(p_n/q_n)$ be the sequence of the convergents of $\xi=\log_d(\alpha)$ and  consider the sequence $(D_n(M))_n$ given for $n\geq 0$ by
$$D_n(M)=\xi+\log_d\Bigl(1+\dfrac{\beta(d-1)}{\alpha M}\Bigr)-\dfrac{p_n}{q_n}.$$ Assume  that there exists a cycle $\Omega$ associated to  the admissible 
triplet $(d,\alpha,\beta)_{\pmb{\pm}}$ with $\min(\Omega)\geq M$, where $M$ is a positive integer large enough such that
$D_{1}(M)<0$.  Then there exists an integer $n_0$ such that the length $L$  of $\Omega$   satisfies
\begin{equation}\label{p2n0+1}
	L=\#\Omega \geq p_{2n_0+1}.
\end{equation}
\end{theorem}
	
\begin{proof}
	According to the inequality \eqref{InequalityCardCycle2} in Lemma~\ref{LemmaexistenceCycle} and as 
	$\min(\Omega)\geq M$, it follows that
	\begin{equation}\label{encadrement de xi-1}
		\xi< \dfrac{L}{\overline{K}}\leq \xi+\log_d\Bigl(1+\dfrac{\beta(d-1)}{\alpha\;\min(\Omega)}\Bigr)\leq 
		\xi+\log_d\Bigl(1+\dfrac{\beta(d-1)}{\alpha\;M}\Bigr).
	\end{equation}
	Now, according to \eqref{rapportSequenceOfConvergents}, for all integer $n\geq 0$, we have $D_{2n}(M)\geq 0$ and the sub-sequence  $(D_{2n}(M))_{n\geq 0}$ is decreasing.   As $M$ is large enough such that $D_{1}(M)<0$ and the sub-sequence  $(D_{2n+1}(M))_{n\geq 0}$ is increasing,
	then there exists an integer $n_0\geq 0$ such that 
	$D_1(M)<D_3(M)<\ldots <D_{2n_0-1}(M)<0$ and  $D_{2n_0+1}(M)>0$ which 
	with  \eqref{rapportSequenceOfConvergents} leads to the following inequalities
	\begin{equation}\label{inequalityInTheorem4.4}
		\dfrac{p_{2n_0+2}}{q_{2n_0+2}}<\xi<\dfrac{p_{2n_0+1}}{q_{2n_0+1}}<\xi+\log_d\Bigl(1+\dfrac{\beta(d-1)}{\alpha M}\Bigr)<\dfrac{p_{2n_0-1}}{q_{2n_0-1}}.
	\end{equation}
Taking into account of inequalities \eqref{encadrement de xi-1} and \eqref{inequalityInTheorem4.4}, it follows that either 
	$\dfrac{p_{2n_0+2}}{q_{2n_0+2}}<\dfrac{L}{\overline{K}}<\dfrac{p_{2n_0+1}}{q_{2n_0+1}}$
	or 
	$\dfrac{L}{\overline{K}}=\dfrac{p_{2n_0+1}}{q_{2n_0+1}}$ or
	$\dfrac{p_{2n_0+1}}{q_{2n_0+1}}< \dfrac{L}{\overline{K}}<\dfrac{p_{2n_0-1}}{q_{2n_0-1}}$. 
	Then, $\dfrac{L}{\overline{K}}$ is an intermediate fraction  between  Farey pair and  according to the Lemma~\ref{Eliahou-Richards-Lemma}, it follows that  either 
	$L\geq p_{2n_0+2}+p_{2n_0+1}$ or $L\geq p_{2n_0+1}$ or $L\geq p_{2n_0-1}+p_{2n_0+1}$.   So,  in each case, we have
	$L\geq p_{2n_0+1}$.\hfill\end{proof}

Note that a similar theorem to the previous one has already been given by Eliahou in \cite{Eliahou1993} for the case of the classical  triplet $(2,3,1)_+$. For  Theorem~\ref{Main_theorem_4.4}, we have as in \cite{Eliahou1993} also used Farey fractions. Apart from this common technique, the two methods are different. First, our result is given in the case of a general admissible  triplet. Furthermore, it allows us to derive an algorithm that gives the lower  bound length directly without any analysis, see the next section for the algorithm. Here, we used the sign of the  sequence $D_{2n+1}$ which is first negative and becomes positive at the step $2n_0+1$ corresponding to the required value $p_{2n_0+1}$, see the algorithm~2 in the following section.

\section{Experimental tests  }\label{ExperimentalSection}
In this section, we will make some experimental tests  illustrating some theoretical results given in the previous sections. For instance, We will discuss briefly  computer verification and validity  of the main conjecture~\ref{MainConjecture2p2q} and its special version the 
conjecture~\ref{MainConjecture2p}.  A more detailed verification is under investigations and will be presented in further work. We will also give  numerical  tests to obtain a lower bound length of an eventual cycle by using Theorem~\ref{Main_theorem_4.3} and Theorem~\ref{Main_theorem_4.4}.      Our  computations were carried out using different environments: Python, SageMaths, Mathematica and Matlab. We have used 
an Intel(R) Core i9-CPU\symbol{64}1.80GHz (8 cores) computer with 16 GB of RAM. 

\subsection{\bf First few  verifications}
We have performed some verifications on the validity of the conjecture~\ref{MainConjecture2p2q} .  The  numerous tests we have performed indicate the validity of this conjecture  for $p$  going from $0$ up to $25$, for $q$ from $0$ to $p$ and 
 for $n$  from $1$ up to $10^{7}$. We recall that the case $p=q =0$ corresponds to the  classical  triplet  $(2,3,1)_{\pmb{+}}$ and it has been verified experimentally with a computer by many authors. For instance,  in  \cite{Oliviera2010}, the author claims that he verified in $2009$ the classical Collatz conjecture  up to 
$2^{62.3}\simeq 5.67\times 10^{18}$ and in  
\cite{Barina2021}, the author claims that he verified in $2025$   the same conjecture  up to  $2^{71}\simeq 2.61\times 10^{21}$. The case $p=3$ and  $q=1$, in the  conjecture~\ref{MainConjecture2p2q},  corresponds to the triplet  $(10,12,8)_{\pmb{+}}$.  We have conjectured in this case that $(10,12,8)_{\pmb{+}}$ is an admissible strong triplet with  the trivial cycle  $\Omega(4)=(4\longrightarrow 8\longrightarrow 16\longrightarrow 24\longrightarrow32\longrightarrow 40\longrightarrow 4)$ as a unique cycle  without any divergent trajectory, i.e.,  for any integer $n\geq 1$, there exists an integer $k\geq 0$ such that $T^{(k)}(n)=2^{3-1}=4$.  We have also verified this affirmation in  special  case  up to
$5\times 10^{11}\simeq 1.82\times 2^{38}$. The speed we have reached was about $4.6\times 10^{6}$ numbers per second,  by  vectorizing    computations in {\tt Matlab}.  This took about  $30$ hours. All the trajectories  starting with  $n\leq 1.82\times 2^{38}$
enter the trivial cycle $\Omega(4)$. 
On the other hand, we may assert that the conjecture~\ref{MainConjecture2p2q}  for $p=3$ and $p=1$   is true  up to  $2^{38}$.  
Of course, we have to improve our speed and our used techniques  in order to increase the limit of  verification. This is our goal for a further work, in which all the techniques and used algorithms will be explained in details. The tests were also coded in Python, SageMaths, Mathematica using $3$ computers in parallel.

\subsection{\bf Tests for  Lower bound length}

Algorithm~\ref{Algo_Therem4.3} is based on Theorem~\ref{Main_theorem_4.3}, it  computes the minimal lower bound length of an  eventual trivial cycle with  a known minimal element.  This algorithm is based on Theorem~\ref{Main_theorem_4.3}.  Suppose that a cycle $\Omega$ associated to the admissible triplet  $(d,\alpha,\beta)_{\pm}$ exists with its minimal element satisfying $\min(\Omega)\geq M$. Then, according to Theorem~\ref{Main_theorem_4.3}, the inequality 
$\displaystyle \#\Omega\geq 
R_{\infty}:=\max_{n\geq 1} R_n(M)$, holds.  The sequence $(R_n(M)_{n\geq 1}$ is a stationary sequence, more precisely,   there exists an integer $n_0$ such that for all $n\geq n_0$, we have $R_n(M)=1$.

\begin{algorithm}[thbp!]
	\caption{Algorithm based on Theorem~\ref{Main_theorem_4.3}}
	\label{Algo_Therem4.3}
	\begin{algorithmic}[1]
		\STATE{{\bf Input:}  The integers $d$, $\alpha$, $\beta$, $M$ (with the  minimal element of the cycle $min(\Omega)\geq M$ ).}
		\STATE{{\bf Output}  The  integer  $R_{\infty}$}
		\STATE{  $\gamma=\alpha\log(d)/(\beta (d-1))$; \;\; $ \xi=\log(\alpha)/\log(d)$;\;\; $x=\xi$}
		\STATE  $(p_0,q_0,p_1,q_1)=(0,1,1,0)$
		\STATE  $(R_0,R)=(0,1)$
		\WHILE{$R_0\leq R $} 
		\STATE $R_0= R$
		\STATE{$a =\lfloor x\rfloor$}
		\STATE{$p = ap_1+p_0$;\;\; $q=aq_1+q_0$;}
		\STATE{  x=1/(x-a)}
		\STATE{$F =p/q$;}
		\STATE{$Q= \gamma\,M/(q_1+q)$}
		\STATE{$R= \Bigl\lfloor \min(q,Q) \Bigr\rfloor+1$;}
		\STATE{$(p_0,q_0) = (p_1,q_1)$;\;\; $(p_1,q_1) = (p,q)$;}
		\ENDWHILE
		\RETURN $R$
	\end{algorithmic}
\end{algorithm} 

Algorithm~\ref{Algo_Therem4.4} is based on Theorem~\ref{Main_theorem_4.4}. It
also computes the minimal lower bound of the cardinality of an  eventual trivial cycle with  a known minimal element.

\begin{algorithm}[thbp!]
	\caption{Algorithm based on Theorem~\ref{Main_theorem_4.4}}
	\label{Algo_Therem4.4}
	\begin{algorithmic}[1]
		\STATE{{\bf Input:}  The integers $d$, $\alpha$, $\beta$, $M$ (with  $min(\Omega)\geq M$ ).}
		\STATE{{\bf Output}  The  integer   $p_{2n_0+1}$ (given in \eqref{p2n0+1}.}
		\STATE{$\xi = \log(\alpha)/\log(d)$;\;\;$x=\xi $}
		\STATE  $(p_0,q_0,p_1,q_1)=(0,1,1,0)$
		\STATE  $(D_0,D)=(1,-1)$
		\WHILE{$D_0* D\leq 0 $} 
		\STATE $D_0= D$
		\STATE{$a =\lfloor x\rfloor$}
		\STATE{$p = ap_1+p_0$; $q=aq_1+q_0$;}
		\STATE{  x=1/(x-a)}
		\STATE{$F =p/q$;}
		\STATE{$D = \xi+\log_d\Bigl(1+\dfrac{\beta(d-1)}{\alpha M}\Bigr)-F$;}
		\STATE{$(p_0,q_0) = (p_1,q_1)$;\;\;$(p_1,q_1) = (p,q)$;}
		\ENDWHILE
		\RETURN   $p$
	\end{algorithmic}
\end{algorithm} 

\begin{center}
	\begin{figure}[bth]
		\centering
		\includegraphics[width=10cm,height=5cm]{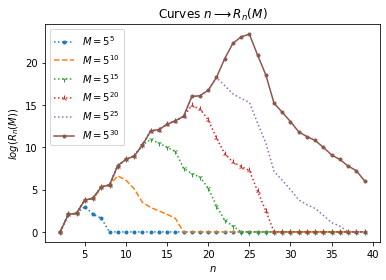}\label{Figure1}
		\caption{\small The curves $n\rightarrow \log(R_n(M))$ for
			different value of $M\in \mathbb{M}$  corresponding to  the triplet
			$(5,6,4)_{\pmb{+}}$.
		}	
	\end{figure}
\end{center}

We have implemented the  Algorithm~\ref{Algo_Therem4.3} and Algorithm~\ref{Algo_Therem4.4}
in {\tt Python} and tested them  in the case of the admissible   triplet
$(5,6,4)_{\pmb{+}}$. In Table~\ref{tab_R_inf}, we  have reported the values of $R_{\infty}(M)$ for different values of $M\in \mathbb{M}=\{5^{5}, 5^{10}, 5^{15}, 5^{20}, 5^{25}, 5^{30}\}$,  assuming that $\min(\Omega)\geq M$, respectively.  The integer $n_0$ corresponds  to the step for which the maximum value is reached. In Figure~1, we have plotted the curves $n\longrightarrow \log(R_n(M))$ for the different values of $M\in \mathbb{M}$.  The curves illustrate the expected behavior of the sequence $(R_n(M))$. Actually, 
we observe  that the  sequence is first increasing until reaching its maximum $R_\infty(M)$ and after it decreases to $1$.  So, for instance, if there exists a non-trivial cycle of minimal element $\min(\Omega)\geq M=5^{30}$, then  its length must satisfy  $\#\Omega \geq  15\,032\,816\,36$. In Table~\ref{tabSub5.2_Rn}, we have given  the results obtained by running the Algorithm~\ref{Algo_Therem4.3}  for $M=M_1:=5^{15}$, $M=M_2:=5^{20}$ and $M=M_3:=5^{25}$.  The  $R_\infty(M)$ is   the boxed value in the column of the $R_n(M)$.

The results from Algorithm~\ref{Algo_Therem4.4}  are summarized in Table~\ref{tabSub5.3} where we give
the value of $p_{2n_0+1}$ for each value of $m\in\mathbb{M}$. We have also reported,  in Table~\ref{tabSub5.3_Dn}, all the values of the sequence $D_n(M)$ given by this algorithm for $M=5^{15},5^{20}, 5^{25}$. We observe in this table that $D_{2n+1}$ changes the sign at $2n_0+1$ step, thus we have $\#\Omega\geq p_{2n_0+1}$. The  $D_{2n_0+1}$  is the boxed value in the column of the $D_n(M)$.

We have observed that  numerically the result from Algorithm~\ref{Algo_Therem4.4} 
does not change for $M\geq 5^{25}$, then the $\#\Omega\geq \max(R_{\infty},p_{2n_0+1})=R_{\infty}$ for $M\geq 5^{25}$. However,  Algorithm~\ref{Algo_Therem4.3} continues to give different results for large values of $M$. For instance, for $\min(\Omega)\geq M=5^{60}$, we get the value  $(R_{\infty} =869802559919868084225$  at the iteration $n_0=40$.

\begin{table}[!t]
	\centering
	\renewcommand{\arraystretch}{1.1}
	\scalebox{1}
	{\begin{tabular}{c|cccccc}
			\hline
			$M$   &$5^{5}$& $5^{10}$ &$5^{15}$&$5^{20}$  &$5^{25}$&$5^{30}$\\
			\hline
			$n_0$&                3                       & 7&             11&16&20&23\\
			\hline
			$R_\infty(M)$& 36      &   2134& 102678&5905570&208606372&15032816369\\
			\hline
	\end{tabular}}
	\caption{\small  The results obtained with $min(\Omega)\geq  M$ with with $M\in A_2$ for the triplet $(5,6,4)_{\pmb{+}}$.
	}
	\label{tab_R_inf}
\end{table} ${}$

\begin{table}[!t]
	\centering
	\renewcommand{\arraystretch}{1.1}
	\scalebox{1}
	{\begin{tabular}{c|cccccc}
			\hline
			$M$   &$5^{5}$& $5^{10}$ &$5^{15}$&$5^{20}$  &$5^{25}$&$5^{30}$\\
			\hline
			$n_0$&               2                       & 3&            5&8&15&15\\
			\hline
			$p_{2n_0+1}$& 226       &   2791&  167863&10850489 &4567472300430581 &4567472300430581 \\
			\hline
	\end{tabular}}
	\caption{\small  The results obtained with $min(\Omega)\geq  M$ with with $M\in A_2$ for the triplet $(5,6,4)_{\pmb{+}}$.
	}
	\label{tabSub5.3}
\end{table}

\begin{table}[!t]
	\centering
	\renewcommand{\arraystretch}{1.1}
	\scalebox{1}
	{\begin{tabular}{r|r|lrrr}
			\hline
			$n$    & $p_n$ &$q_n$ & $R_n(M_1)$ &$R_n(M_2)$&$R_n(M_3)$  \\
			\hline		
			0&                                  1&1                   & 1& 1&1\\
			1&                                 9&8                    &8 &8&8\\
			2&                                 10&9                   &           9&9&9\\
			3&                                 49&44                &             44&44&44\\
			4&                                 59&53                                 &             53&53&53\\
			5&                                226&203                                &             203&203&203\\
			6&                                285&256                                &            256&256&256\\
			7&                               2791&2507                               &            2507&2507&2507\\
			8&                               5867&5270                               &           5270&5270&5270\\
			9&                               8658&7777                               &           7777&7777&7777\\
			10&                              31841&28601                              &           28601&28601&28601\\
			11&                             167863&150782                             &          \fbox{102678}&150782&150782\\
			12&                             199704&179383                             &           55786&179383 &179383 \\
			13&                             367567&330165                             &          36147&330165 &330165 \\
			14&                             567271&509548                             &          21935&509548  &509548  \\
			15&                             934838&839713                             &          13651&839713&839713 \\
			16&                            9915651&8906678                            &           1890&\fbox{5905570}&8906678 \\
			17&                           10850489&9746391                            &            988&3085712&9746391 \\
			18&                           20766140&18653069                           &            649&2026729&18653069 \\
			19&                           93915049&84358667                           &            179&558752&84358667 \\
			20&                          866001581&777881072                          &             22&66755&\fbox{208606372}\\
			\hline
	\end{tabular}}
	\caption{\small  The results obtained with $min(\Omega)\geq  M_i$ with with $M_1= 5^{15}$, $M_2=5^{20}$, $M_3=5^{25}$ for 
		the triplet $(5,6,4)_{\pmb{+}}$.
	}
	\label{tabSub5.2_Rn}
\end{table} 

\subsection{Verification of the  lower bound in Theorem~\ref{Lower_Hurwitz_Bound}}
 In \cite{Barina2021} and more recently in 2025, the author verified the classical Collatz conjecture for all integer $n\leq M_0=2^{71}$.  According to Theorem~\ref{Lower_Hurwitz_Bound}, if an hypothetical  cycle $\Omega$ exists with $\min(\Omega)>M_0$, then the length $L$ of $\Omega$ must satisfies $L\geq \mu_0\sqrt{M_0}\approx 46\,859\,289\,878$. While with Theore~\ref{Main_theorem_4.4}, we obtain $L\geq p_{23}=217\,976\,794\,617$.

\section{Conclusion}
In this paper, we have introduced and investigated a simple and unified extension of the classical Collatz problem. Several general theorems have been established, and new conjectures have been formulated (Conjectures~\ref{MainConjecture2p2q}- \ref{MainConjecture2p} ), providing a broader framework that includes the original Collatz conjecture as a particular case.
Furthermore, 
we have examined lower bounds for the cardinality of possible cycles, leading to efficient algorithms for their computation. Numerical experiments and computational verifications have been carried out to illustrate and support our conjectures. Further investigations are currently underway to obtain deeper theoretical results and  numerical tests.

\begin{table}[!t]
	\centering
	\renewcommand{\arraystretch}{1.1}
	\scalebox{1}
	{\begin{tabular}{r|r|lrrr}
			\hline
			$n$    & $p_n$ &$q_n$ & $D_n(M_1)$ &$D_n(M_2)$&$D_n(M_3)$  \\
			\hline
			
			0&                                  1&1                   & 1.13e-1& 1.13e-1&1.13e-1 \\
			1&                                 9&8                    &-1.17e-2 &-1.17e-2 & -1.17e-2 \\   
			2&                                 10&9                   & 2.17e-3&2.17e-3 & 2.17e-3\\
			3&                                 49&44                & -3.54e-4&-3.54e-4&-3.54e-4\\
			4&                                 59&53                                 & 7.52e-5& 7.52e-5 &7.52e-5\\
			5&                                226&203                                &             -1.77e-5&-1.77e-5&-1.77e-5\\
			6&                                285&256                                &            1.50e-6 & 1.50e-6 &1.50e-6 \\
			7&                               2791&2507                               &            -5.55e-8&-5.56e-8&-5.56e-8\\
			8&                               5867&5270                               &           2.02e-8&2.01e-8&2.01e-8\\
			9&                               8658&7777                               &           -4.23e-9&-4.29e-9&-4.29e-9\\
			10&                              31841&28601                              &            2.62e-10&2.08e-10&2.08e-10\\
			\fbox{11}&                             \fbox{167863}&150782                             &          \fbox{3.05e-11}&-2.38e-11&-2.38e-11\\
			12&                             199704&179383                             &          6.74e-11& 1.32e-11 &1.32e-11\\
			13&                             367567&330165                             &         5.06e-11&-3.72e-12 &-3.73e-12\\
			14&                             567271&509548                             &          5.65e-11 &2.23e-12 &2.21e-12\\
			15&                             934838&839713                             &          5.42e-11& -1.10e-13&-1.27e-13\\
			16&                            9915651&8906678                            &          5.43e-11    &2.40e-14&6.64e-15\\
			\fbox{17}&                           \fbox{10850489}&
			9746391&       5.43e-11     &\fbox{1.25e-14} &-4.88e-15\\
			18&                           20766140&18653069                           &            5.43e-11&1.80e-14&6.21e-16\\
			19&                           93915049&84358667                           &         5.43e-11&1.74e-14&-1.50e-17 \\
			20&                          866001581&777881072                          &             5.43e-11&1.74e-14&2.54e-19 \\
			21&                         5289924535&4751645099                         &              5.43e-11&1.74e-14&-1.65e-20\\
			22&                        11445850651&10281171270                        &              5.43e-11&1.74e-14&3.93e-21\\
			23&                        16735775186&15032816369                        &              5.43e-11&1.74e-14&-2.54e-21\\
			24&                        28181625837&25313987639                        &              5.43e-11&1.74e-14&8.82e-23\\
			25&                       495823414415&445370606232                       &              5.43e-11&1.74e-14&-5.00e-25\\
			26&                      4986415769987&4479020049959                      &              5.43e-11&1.74e-14&1.63e-27\\
			27&                    150088296514025&134815972105002                    &              5.43e-11    &1.74e-14&-2.35e-29\\
			28&                    305163008798037&274110964259963                    &              5.43e-11&1.74e-14&3.50e-30\\
			29&                   1065577322908136&957148864884891                    &              5.43e-11&1.74e-14&-3.00e-31\\
			30&                   3501894977522445&3145557558914636                   &              5.43e-11&1.74e-14&1.00e-31 \\
			\fbox{31}&                   \fbox{4567472300430581}&4102706423799527                   &              5.43e-11&1.74e-14&\fbox{0.00e-29 }\\
			\hline
	\end{tabular}}
	\caption{\small  The results obtained with $min(\Omega)\geq  M_i$ with with $M_1= 5^{15}$, $M_2=5^{20}$, $M_3=5^{25}$for 
		the triplet $(5,6,4)_{\pmb{+}}$.
	}
	\label{tabSub5.3_Dn}
\end{table} 

\newpage


\end{document}